\documentclass[11pt]{amsart}
\usepackage{graphicx} 
\usepackage[utf8]{inputenc}
\usepackage[a4paper, lmargin=1in, rmargin=1in]{geometry}
\usepackage{amsmath, amssymb, amsthm, bbm}
\usepackage{comment}
\usepackage{mathrsfs}
\usepackage{mathalpha}
\DeclareMathAlphabet\mathzapf{T1}{pzc}{mb}{it}
\usepackage{bbm}
\usepackage{float}
\usepackage{pgfplots}
\pgfplotsset{my style/.append style={axis x line = middle, axis y line = middle, xlabel={$\theta$}, axis equal}}
\pgfplotsset{compat=1.18}
\usepackage{appendix}
\usepackage{caption}
\usepackage[subpreambles=false]{standalone}
\usepackage{import}
\usepackage{mathtools}
\theoremstyle{plain}
\newtheorem{theorem}{Theorem}[section]
\newtheorem{proposition}[theorem]{Proposition}
\newtheorem{lemma}[theorem]{Lemma}
\newtheorem{corollary}[theorem]{Corollary}
\newtheorem{remark}[theorem]{Remark}
\usepackage[
backend=biber,
sorting=nyt
]{biblatex}
\usepackage{hyperref}
\usepackage{xurl}
\hypersetup{breaklinks=true}
    
\title{On the average least negative Hecke eigenvalue}
\author{Jackie Voros}
\email{jackie.voros@bristol.ac.uk}

\begin{document}

\begin{abstract}
We show that the first sign change of Hecke eigenvalues of classical newforms has a finite mean, which we also compute. We distinguish between the first negative \textit{prime} Hecke eigenvalue, and the first negative Hecke eigenvalue. This problem can be considered to be an analogue of the least quadratic non-residue problem, of which the average was explored by Erd\H{o}s in 1961. In fact, the average least negative prime Hecke eigenvalue has the same value as the average least quadratic non-residue, under GRH. To compute these averages, we develop large sieve inequalities that are uniform in both the weight and level aspect.


\end{abstract}

\maketitle

\section{Introduction and statement of results}
Let $H_k^*(N)$ denote the finite set of newforms of weight $k \geq 2$ for $\Gamma_0(N)$, where $k$ is an even integer, and $N \geq 1$ is an integer. Then $f \in H_k^*(N)$ is a normalised cuspform that is new at level $N$, and is also a Hecke eigenform. We restrict to trivial character to ensure that the Hecke eigenvalues are real. Hence $f$ has the Fourier series expansion,

\begin{equation*}
    f(z) = \sum_{n = 1}^{\infty} \lambda_f(n) n^{(k-1)/2}e(nz), \quad e(z)=e^{2\pi i z}, \quad \lambda_f(1)=1,
\end{equation*}

where the Hecke eigenvalues $\lambda_f(n)$ have been normalised so that Deligne's bound implies $|\lambda_f(p)| \leq 2$.

The sign changes of Hecke eigenvalues is a topic of great interest. Knopp, Kohnen and Pribitkin \cite{KnoKohPri} have shown that any non-trivial cusp form with real coefficients has infinitely many sign changes. This can be seen from Landau's classical theorem on Dirichlet series, and the properties of the associated $L$-functions.

Indeed we now have stronger assertions on properties of the sign changes of $\lambda_f(n)$ in the form of multiplicity one theorems. Kowalski, Lau, Soundararajan and Wu \cite[Theorem 4]{klsw} show that the signs of the Hecke eigenvalues determine the eigenform itself. This result stems from the observation that for some weight $k$, level $N$ newform, there is asymptotically no bias in the signs of its eigenvalues.

The question of where the first sign change of $\lambda_f(n)$ occurs has been studied by several authors. Let $n_f$ be the least integer such that $\lambda_f(n_f) < 0$. An upper bound on $n_f$ was investigated first by Kohnen and Sengupta \cite{KohSen}. With Iwaniec \cite{IwaKohSen} they later improved their result, and Kowalski, Lau, Soundararajan and Wu \cite{klsw} introduced new ideas, which Matom\"aki \cite{Mat} optimised, to get the bound

\begin{equation}
    \label{eq:nf}
    n_f \ll Q^{3/8},
\end{equation}

where $Q:=k^2N$ is defined as the analytic conductor of $f$. This upper bound is so far the best known, although it seems far from optimal. The methods used above are connected with the zero-free region of the associated $L$-function, $L(s,f)$. As mentioned in the papers referenced above, assuming the Generalised Riemann Hypothesis (GRH) for $L(s,f)$, one can show

\begin{equation}
    \label{eq:nfGRH}
    n_f \ll (\log Q)^2.
\end{equation}

We think of this question as being analagous to the least quadratic non-residue problem, and we can clearly see this when considering the conditional estimate of the least quadratic non-residue, see Linnik \cite{Lin} and Ankeny \cite{Ank}. However, we cannot make such an immediate parallel due to the difference in arithmetic structure of Hecke eigenvalues and the Legendre symbol: the Legendre symbol is totally multiplicative, while the Hecke eigenvalues are only multiplicative.

For this reason, we also consider the least negative \textit{prime} Hecke eigenvalue, $p_f$. That is, the least prime $p_f$ such that $\lambda_f(p_f) < 0$. As will be discussed in Section \ref{sec:motivation}, we actually find that we have the same bound as in \eqref{eq:nfGRH} for $p_f$ under GRH -- although it should be noted, as with the least quadratic non-residue, that the true value of both $n_f$ and $p_f$ is expected to be less than what can be shown via GRH.

It is certainly not necessary that $n_f$ and $p_f$ coincide. Just from the definition we can see that $n_f \leq p_f$, but in practice one can find that we have a strict inequality in many cases. In this paper we explore this relationship and compute the average value of $p_f$ and $n_f$.

\begin{theorem}
\label{thrm:averagepf}
Let $p_i$ denote the $i^{th}$ prime, and let $p_f$ be the least prime such that $\lambda_f(p_f) < 0$. Assuming GRH we have,

    \begin{equation*}
        \lim_{k + N \to \infty} \frac{1}{|H_k^*(N)|} \sum_{f \in H_k^*(N)} p_f = \sum_{i=1}^{\infty} \frac{p_i}{2^i},
    \end{equation*}
    
as $k+N$ tends to infinity over even $k$ and square-free $N$.
\end{theorem}

The GRH condition is on each symmetric power $L$-function, $L(s, \mathrm{Sym}^mf)$. That is, for $m \geq 1$ the non-trivial zeros of $L(s, \mathrm{Sym}^mf)$ occur on $\Re(s)=1/2$. We require this as we do not have a strong enough upper bound on $p_f$ -- see Section \ref{sec:motivation}. We also discuss further in Section \ref{sec:motivation} the structure of eigenvalues at squares of primes that divide the level, and hence why we require square-free $N$.

\begin{theorem}
\label{thrm:averagenf}
Let $p_i$ denote the $i^{th}$ prime and let $n_f$ be the least integer such that $\lambda_f(n_f) < 0$. Let $(k_r, N_r)$ denote a pair of positive integers indexed by $r \geq 1$, where for each $r$ we have either,

\begin{itemize}
    \item $k_r$ is even;
    \item $N_r$ only has prime factors strictly greater than $\log(k_r N_r)$;
    \item $N_r$ tends to infinity as $r$ does;
\end{itemize}

or we have,

\begin{itemize}
    \item $N_r = 1$ for all $r \geq 1$;
    \item $k_r$ is even and tends to infinity as $r$ does.
\end{itemize}

Then,
\begin{equation}
\label{eq:nf limit}
    \lim_{r \to \infty} \frac{1}{|H_{k_r}^*(N_r)|} \sum_{f \in H_{k_r}^*(N_r)} n_f = \sum_{i \geq 1} \sum_{n \geq 1} p_i^n \prod_{j = 1}^{\pi(p_i^n)} \mu_{p_j}(I_{p_i^n}(p_j)),
\end{equation}

Here, $\mu_p$ denotes the $p$-adic Plancherel measure,

\begin{equation*}
    \mu_p = \frac{2}{\pi} \left( 1 + \frac{1}{p} \right) \frac{\sin^2 \theta}{(1-p^{-1})^2 + \frac{4}{p}\sin^2 \theta} \textrm{d} \theta,
\end{equation*}

and is applied at intervals $I_{p_i^n}(p_j) \subseteq [0, \pi]$:

\begin{equation*}
    I_{p_i^n}(p_j) = \begin{cases}
        \Big[ 0, \frac{\pi}{a_{p_i^n}(p_j)+1} \Big] &\textrm{ if } p_j \neq p_i, \\
            
            \Big[ \frac{\pi}{n+1}, \frac{\pi}{n} \Big] &\textrm{ if } p_j=p_i,
        \end{cases}
\end{equation*}
     with $a_{p_i^n}(p_j) = \lfloor n\log p_i / \log p_j \rfloor$ being the greatest power such that $p_j^{a_{p_i^n}(p_j)} < p_i^n$.
\end{theorem}

We do not require GRH in this theorem as we have a strong enough unconditional bound on $n_f$, \eqref{eq:nf}. In this theorem we may let our weight $k$ either be fixed or vary arbitrarily, provided $k$ is even, but we must have that the level $N$ tends to infinity over integers with large enough prime factors. On the other hand, we can also fix $N=1$, provided $k$ tends to infinity over even integers. This is due to the structure of eigenvalues at primes that divide the level, see Section \ref{sec:motivation}.

We remark that it is possible to remove the restriction that $N$ requires large prime factors, via similar methods in this paper, if one took into consideration the structure of eigenvalues at primes that divide the level. However, the average would not be the same as in Theorem \ref{thrm:averagenf}, and further a limiting average may not exist if $N$ tends to infinity arbitrarily.

\begin{remark}
\label{rem:comp}
    Both averages in Theorems \ref{thrm:averagepf} and \ref{thrm:averagenf} are finite, with values that have been computed as,

\begin{equation*}
    \sum_{i=1}^{\infty} \frac{p_i}{2^i} \approx 3.674643966011328 \dots
\end{equation*}

\begin{equation*}
    \sum_{i \geq 1} \sum_{n \geq 1} p_i^n \prod_{j = 1}^{\pi(p_i^n)} \mu_{p_j}(I_{p_i^n}(p_j)) \approx 2.9423403000531483 \dots 
\end{equation*}
\end{remark}

The key point in both statements concerns asking the question: what is the likelihood that $p_f$ or $n_f$ takes a certain value in terms of the weight and level of the form? As we will see in Section \ref{sec:motivation} we are familiar with the distribution of eigenvalues in two ways: at a fixed prime and as the weight and level tend to infinity, or at a fixed weight and level as the prime value tends to infinity. For Theorems \ref{thrm:averagepf} and \ref{thrm:averagenf} we are tasked with combining these two distributions, which is in fact not particularly complicated.

Indeed the main bulk of the proofs lies in showing that $p_f$ or $n_f$ is unlikely to be large. For this we use large sieve techniques. Originally derived by Linnik \cite{Lin2} in the context of the least quadratic non-residue, there have since been analogous statements developed for modular forms, notably by Deshouillers and Iwaniec \cite{DesIwa}.

However many of the large sieve inequalities for cusp forms found in the literature were either not strong enough for our needs or not uniform in both the weight and the level aspect, particularly for Theorem \ref{thrm:averagenf} where we wanted to avoid using the GRH condition. This led to us developing a version of the large sieve inequality found in \cite{DesIwa} that is stronger in some ranges.

\begin{theorem}
\label{thrm:weighted linnik sieve}
    Let $a_m$ be complex numbers for all $m \leq M$, for some $M > 1$. Suppose that $Nk^{\alpha} \geq 2\pi M n$ for $1/2 < \alpha < 1$ with $n = e^{k/(k-3)}$ and $k > 2$ is even. Then we have,

    \begin{equation*}
        \sum_{f \in H_k^*(N)} \omega_f \left| \sum_{\substack{ m \leq M \\ m \nmid N}} a_m \lambda_f(m) \right|^2 \ll \left( 1 + \frac{M}{Nk^{\eta}} \right) \| a\|^2
    \end{equation*}

    where $\eta = k(1-\alpha) - k^{2\alpha - 1}/2$ and $\omega_f$ is the harmonic weight:

    \begin{equation*}
        \omega_f = \frac{\Gamma(k-1)}{(4\pi)^{k-1} \langle f, f \rangle},
    \end{equation*}
    
    with $\langle \cdot, \cdot \rangle$ the Petersson inner product. In the $k=2$ case an extra factor of $\log M$ is applied to the second term, and we have $n=1$. The implied constant is absolute.
\end{theorem}

In fact, this theorem also holds for $0 < \alpha \leq 1/2$, but is not proven here, see Remark \ref{rem: alpha}. We remove the conditions on the length of the sieved interval to give a complete sieve that has less saving in $k$, although a version of this has already been shown by Lam \cite[Theorem 2.2]{Lam} but with a slightly weaker $N^{-1+\varepsilon}$ in place of $N^{-1}$ (the saving in $k$ is the same), see Theorem \ref{thrm:complete linnik sieve}.

With Theorem \ref{thrm:weighted linnik sieve} we can show that $p_f$ and $n_f$ are extremely unlikely to be large. In particular, we have,

\begin{theorem}
\label{thrm: linnik nf}
    Let $(\omega_p)$ be a sequence of signs indexed by primes $p \leq (\log kN)^A$, where $A>1$. Let $k$ be an even integer, and $N>1$. Then we have,

    \begin{equation*}
        \# \{ f \in H_k^*(N): \lambda_f(p)\omega_p > 0, \  p\leq (\log kN)^A, \  p \nmid N\} \ll_{\varepsilon_1, \varepsilon_2} N^{1/2 + 1/2A + \varepsilon_1}k^{1/A + \varepsilon_2},
    \end{equation*}
    
    for any $\varepsilon_1, \varepsilon_2 > 0$, and provided $k$ and $N$ are large enough. The implied constant depends only on $\varepsilon_1$ and $\varepsilon_2$.
    
\end{theorem}

This theorem may be adapted to sieve sets that count the number of forms that have $\lambda_f(p) \in I_p$, for $p \leq (\log kN)^A$, where $I_p$ is an interval of the form $[ a, b ] \subseteq [-2, 2]$, see Section \ref{sec:proof of large sieve}.

Taking $\omega_p = +1$ for all $p \leq (\log kN)^A$, then the set in Theorem \ref{thrm: linnik nf} gives us a bound on the number of forms with $p_f > (\log kN)^A$, and also contains the set of forms with $n_f > (\log kN)^A$. We may compare this to results found in \cite[Theorem 2]{klsw} which sieve the set of newforms with $p_f$ or $n_f$ slightly smaller (specifically $> \log kN$), and which we also adapt for our purposes, see Proposition \ref{prop: large sieve}, Corollary \ref{cor: large sieve} and Corollary \ref{cor: nf large sieve}.

We structure this paper in the following way. In Section \ref{sec:motivation} we describe further background on the problem that give motivation for Theorems \ref{thrm:averagepf} and \ref{thrm:averagenf}. In Section \ref{sec:proof thrm pf} we prove these theorems in detail, and in Section \ref{sec:aux proof} we prove auxiliary statements required in the proofs of Theorems \ref{thrm:averagepf} and \ref{thrm:averagenf}. Finally in Section \ref{sec:proof of large sieve} we prove Theorems \ref{thrm:weighted linnik sieve} and \ref{thrm: linnik nf}, along with an auxiliary corollary. In the Appendix we complete the sieve statement in Theorem \ref{thrm:weighted linnik sieve} so that we no longer restrict the length of the sieved sum.

\subsection*{Acknowledgments}
I would like to thank Jonathan Bober, Andrew Booker and Min Lee for their guidance and suggestions.

\section{Motivation}
\label{sec:motivation}
We begin with motivating Theorem \ref{thrm:averagepf}. As mentioned previously, we may consider this problem to be an analogue to the least quadratic non-residue problem. Understanding the size of the least quadratic non-residue is a classical problem in analytic number theory, with origins from Gauss in the late $18^{\text{th}}$ century \cite{Gau}. A detailed exposition on the problem can be found here \cite{Tao}. An alternative, easier problem is to consider the average least quadratic non-residue.

In this direction we have the following result of Erd\H os \cite{Erd}. Let $n_2(p)$ denote the first quadratic non-residue modulo $p$. Then,

\begin{equation*}
    \lim_{x \to \infty} \frac{1}{\pi(x)} \sum_{p \leq x} n_2(p) = \sum_{i=1}^{\infty} \frac{p_i}{2^i},
\end{equation*}

where $\pi(x)$ denotes the prime counting function, and $p_i$ denotes the $i^{th}$ prime. This can be seen heuristically from quadratic reciprocity, from which we have that a given prime number is a residue (or non-residue) modulo $p$ exactly half the time. Thus one can say that the `probability' that $n_2(p) = p_i$ is $2^{-i}$, hence giving the statement above. In proving the statement, quadratic reciprocity is not sufficient to show that this is the average. To complete his statement, Erd\H os used the (new at the time) large sieve methods of Linnik \cite{Lin} and R\'enyi \cite{Ren} to show that $n_2(p)$ cannot be too large.

This result has been generalised in many ways by Burgess and Elliot \cite{BurEll}, Elliot and Murata \cite{EllMur}, with the most recent generalisations by Pollack \cite{Pol}, and Martin and Pollack \cite{MarPol}. These extensions have provided inspiration to extend this problem in other directions, with our main focus on Hecke eigenvalues.

Although Hecke eigenvalues may take values other than $\pm 1$, we may view this problem in terms of their signs instead. For this reason we require the Hecke eigenvalues to be real, and thus there is not such a great analogy for newforms with non-trivial complex character. To explore this further, we state some facts on newforms.

We have Hecke multiplicity,

\begin{equation}
    \label{eq:Heckemult}
    \lambda_f(m) \lambda_f(n) = \sum_{\substack{d | (m,n) \\ (d,N) = 1}} \lambda_f \left( \frac{mn}{d^2} \right),
\end{equation}

for all integers $m, n \geq 1$. We also have Deligne's bound, or the Ramanujan-Petersson conjecture that was proved by Deligne,

\begin{equation*}
    |\lambda_f(n)| \leq \tau(n),
\end{equation*}

where $\tau(n)$ denotes the divisor function. In particular when $n$ is a prime, we have that $|\lambda_f(p)| \leq 2$ and hence we may associate an angle $\theta_f(p) \in [0, \pi]$ such that,

\begin{equation}
\label{eq: thetafp}
    \lambda_f(p) = 2 \cos \theta_f(p).
\end{equation}

The Sato--Tate conjecture provides further information about the distribution of 
$\lambda_f(p)$. Initially conjectured for cusp forms associated with non-CM elliptic curves, it was extended by Serre to all Hecke eigenforms for the full modular group. This extension was proven by Barnet-Lamb, Geraghty, Harris and Taylor \cite{BarGerHarTay} and states for $f \in S_k^{\text{new}}(N)$ (the set of cusp forms of weight $k$ and that are new at level $N$),

\begin{equation*}
    \lim_{x \to \infty} \frac{1}{\pi(x)} | \{ p \leq x : \theta_f(p) \in [\alpha, \beta], p\nmid N \} |  = \mu_{ST}([\alpha, \beta]) = \int_{\alpha}^{\beta} \mathrm{d} \mu_{ST},
\end{equation*}

for some $[\alpha, \beta] \subseteq [0, \pi]$, and where $\mu_{ST}$ is the Sato-Tate measure,

\begin{equation*}
    \mu_{ST} = \frac{2}{\pi} \sin^2 \theta \mathrm{d}\theta.
\end{equation*}

We note here that $\mu_{ST}([0, \pi/2]) = \mu_{ST}([\pi/2, \pi]) = 1/2$ so we have a similar property as that of the Legendre symbol, in that the signs of Hecke eigenvalues should asymptotically be positive and negative an equal amount.

The Sato--Tate distribution only describes eigenvalues at primes that do not divide the level. For the primes that do divide the level, we have (see e.g. \cite[Chapter 9.7]{Kna} or \cite[Chapter 13]{CohStr}),

\begin{equation}
\label{eq:ALL theory}
    \lambda_f(p) = \begin{cases}
        0 &\textrm{ if } p^2 |N \\
        \pm p^{-1/2} &\textrm{ if } p|N \textrm{ but } p^2 \nmid N.
    \end{cases}
\end{equation}

We will see in the proof of Lemma \ref{lem:pf half} that at primes that exactly divide the level, the eigenvalues are distributed uniformly in $\{ -p^{-1/2}, p^{-1/2} \}$, thus the signs of Hecke eigenvalues at all primes have same property as that of the Legendre symbol, provided $N$ is square-free.

Furthermore, these eigenvalues are totally multiplicative,

\begin{equation*}
    \lambda_f(p^2) = \lambda_f(p)^2 \hspace{1em} \text{ for } p|N.
\end{equation*}

So for ease, when considering the average $n_f$ we do not include these eigenvalues.

However, to compute the average $p_f$ it is still necessary to have some upper bound on $p_f$, which is a problem in itself. Indeed, the best known unconditional bound on $p_f$ is from Thorner \cite[Theorem 1.1]{Tho};

\begin{equation}
\label{eq:pf}
    p_f \ll \frac{(kN)^{c}}{\log(kN)},
\end{equation}

for some constant $c$. This constant, and the implied constant, are both shown to be explicitly computable and are likely to be quite large; where `large' means the overall power of $kN$ is greater than 1. Given that we have $|H_k^*(N)| = o(k\varphi(N))$, this bound would contribute too large a value in the average.

The bound \eqref{eq:pf} is achieved through the work of Newton and Thorne \cite{NewThoI, NewThoII} who proved that the $m^{th}$ symmetric power lift, $\mathrm{Sym}^mf$, corresponds with a cuspidal automorphic representation of $GL_{m+1}(\mathbb{A})$ for all $m \geq 1$, where $\mathbb{A}$ is the ring of adeles over $\mathbb{Q}$, and hence implies the analytic continuation and existence of a functional equation for $L(s,\text{Sym}^m f)$. Then with the zero-free regions of $L(s,\textrm{Sym}^mf)$ we have \eqref{eq:pf}, but Thorner \cite[Theorem 1.3]{Tho} also proves that if one assumes GRH for each $L(s,\textrm{Sym}^mf)$, we have the same bound for $p_f$ as in \eqref{eq:nfGRH}, hence the GRH condition in Theorem \ref{thrm:averagepf}.

We may not apply the same logic when considering $n_f$ due to the multiplicativity of $\lambda_f(n)$. While it is true that we still have a similar distribution of signs of $\lambda_f(n)$ for all $n$, in that Matom\"aki and Radziwi\l{}\l{} \cite[Theorem 1.1]{MatRad} show us that half the $\lambda_f(n)$ are positive and half are negative asymptotically, at the local scale we may not make the same statements.

We have that $n_f$ must be a prime power. We can see that $n_f=p^k$ implies a condition on $\theta_f(q)$ for all primes $q \leq p^k$, and on all powers of $q$ such that $q^a \leq p^k$. Therefore we must be more specific in exactly how the $\theta_f(p)$ are distributed for each $p$, particularly for small $p$, as we take the measure of these intervals. To this end, we utilise the following distribution proved by Serre \cite{Ser}, with a version also shown by Conrey, Duke and Farmer for level 1 cuspforms \cite{ConDukFar}. For a fixed prime $p$,

\begin{equation*}
    \lim_{k+N \to \infty} \frac{1}{|H_k^*(N)|} |\{ f \in H_k^*(N) : \theta_f(p) \in [\alpha, \beta] , p \nmid N \}| = \int_{\alpha}^{\beta}  \mathrm{d}\mu_p,
\end{equation*}

for some $[\alpha, \beta] \subseteq [0, \pi]$, and $\mu_p$ denotes the $p$-adic Plancherel measure,

\begin{equation}
    \label{eq: p-adic plancherel}
    \mu_p = \frac{2}{\pi} \left(1+ \frac{1}{p}\right) \frac{\sin^2\theta}{(1-p^{-1})^2+\frac{4}{p}\sin^2\theta} \mathrm{d}\theta.
\end{equation}

We can see that as $p$ tends to infinity, $\mu_p$ tends to $\mu_{ST}$, but for small $p$ we certainly get quite a different distribution, see Figure \ref{fig: mu_p}. Then the shape of the intervals $I_{j,i,n}$ used in Theorem \ref{thrm:averagenf} correspond to intervals that $\theta_f(p_j)$ are in such that $n_f = p_i^n$.

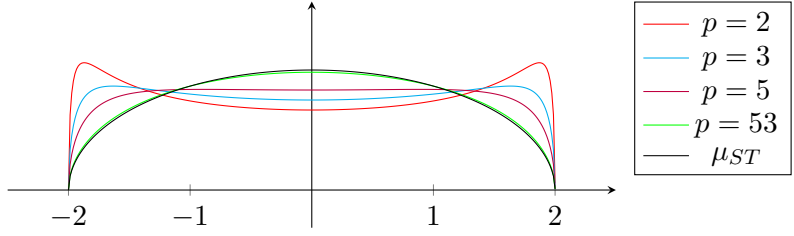
\begin{figure}[H]
\centering
    \begin{tikzpicture}
        \begin{axis}[xtick distance=1, ytick distance=1, xmin=-2.5, xmax=2.5, ymin=-0.1, ymax=0.5, width=8cm, height=3cm, scale only axis, legend pos = outer north east, axis lines = center]
            \addplot[domain=-2:2,samples=300,color=red]{(3/pi)*(sqrt(1-(x^2)/4))/((sqrt(2)+1/(sqrt(2)))^2-x^2)};
            \addlegendentry{$p=2$}
            \addplot[domain=-2:2,samples=300,color=cyan]{(4/pi)*(sqrt(1-(x^2)/4))/((sqrt(3)+1/(sqrt(3)))^2-x^2)};
            \addlegendentry{$p=3$}
            \addplot[domain=-2:2,samples=300,color=purple]{(6/pi)*(sqrt(1-(x^2)/4))/((sqrt(5)+1/(sqrt(5)))^2-x^2)};
            \addlegendentry{$p=5$}
            \addplot[domain=-2:2,samples=300,color=green]{(54/pi)*(sqrt(1-(x^2)/4))/((sqrt(53)+1/(sqrt(53)))^2-x^2)};
            \addlegendentry{$p=53$}
            \addplot[domain=-2:2,samples=300,color=black]{(1/(2*pi))*sqrt(4-x*x)};
            \addlegendentry{$\mu_{ST}$}
        \end{axis}
    \end{tikzpicture}
\caption{Distribution of $\lambda_f(p)$ as $k,N \to \infty$ (coloured lines), in comparison to the distribution of $\lambda_f(p)$ as $p \to \infty$ (black line).}
\label{fig: mu_p}
\end{figure}

\section{Proof of main results}
\label{sec:proof thrm pf}

\subsection{Proof of Theorem \ref{thrm:averagepf}}
\label{Sec:pf proof}

We first consider this problem with a finite set of primes. We adapt a theorem of Serre \cite{Ser} to find the proportion of newforms with eigenvalues that have a certain sign sequence. We note this lemma includes eigenvalues at primes that divide the level. We prove this lemma later, in Section \ref{sec:aux proof}.

\begin{lemma}
\label{lem:pf half}
    Let $\mathcal{P}$ be a finite set of primes, and $(\omega_p)_{p \in \mathcal{P}}$ be a sequence of signs indexed by primes in $\mathcal{P}$. We have,

    \begin{equation*}
     \frac{1}{|H_k^*(N)|} |\{ f \in H_k^*(N): \omega_p\lambda_f(p)>0 \text{ for } p \in \mathcal{P} \} | = \frac{1}{2^{|\mathcal{P}|}}(1+o(1)),
    \end{equation*}  

     as $k + N$ tends to infinity over even $k$ and square-free $N$.
\end{lemma}

Label the set in Lemma \ref{lem:pf half} as $G_{k,N}((\omega), \mathcal{P})$, where $(\omega) = (\omega_p)_{p \in \mathcal{P}}$ is a sequence of signs indexed by primes in $\mathcal{P}$, that is

\begin{equation*}
    G_{k,N}((\omega), \mathcal{P}) := |\{ f \in H_k^*(N): \omega_p\lambda_f(p)>0, p \in \mathcal{P}\} |.
\end{equation*}

Let $\mathcal{P}(z)$ be the set of all primes less than or equal to $z$, and set $(\omega) = (+1, +1, \dots, -1)$. Then we may rewrite our sum of $p_f$ over newforms into a sum in terms of $G$ over primes, given by

\begin{equation*}
    \sum_{f \in H_k^*(N)} p_f = \sum_{p \text{ prime}} p \cdot G_{k,N}((\omega), \mathcal{P}(p)). 
\end{equation*}

Let $|\mathcal{P}(p)|$ tend to infinity as $k$ and $N$ do, say at some rate $A(k,N)$, so that we have,

\begin{equation}
\label{eq:splitsum_pf}
    \sum_{f \in H_k^*(N)} p_f = \sum_{p \leq A(k,N)} p \cdot G_{k,N}((\omega), \mathcal{P}(p)) + \sum_{\substack{f \in H_k^*(N) \\ p_f > A(k,N)}} p_f.
\end{equation}

By Lemma \ref{lem:pf half}, if $A(k,N)$ tends to infinity appropriately slowly then the first sum is, 

\begin{equation}
\label{eq:A(k,N)}
    \lim_{k+ N \to \infty} \sum_{p \leq A(k,N)} p \cdot G_{k,N}((\omega), \mathcal{P}(p))  = \sum_{i=1}^{\infty}(1 + o(1)) \frac{p_i}{2^i}|H_k^*(N)|.
\end{equation}

Hence it remains to show that the second sum over $p_f > A(k,N)$ in \eqref{eq:splitsum_pf} is equal to $o(|H_k^*(N)|)$, or rather since we have that (Proposition \ref{prop:serre}),

\begin{equation}
\label{eq: size Hk*N}
    |H_k^*(N)| \sim \frac{k-1}{12}\varphi(N)
\end{equation}

as $k+ N \to \infty$, it suffices to show the second sum is equal to $o(k\varphi(N))$. We use large sieve techniques, which we will prove in Section \ref{sec:aux proof}.

\begin{corollary}
\label{cor: large sieve}
    For $k\geq 2$ an even integer, $N$ square-free there are absolute positive constants $C_0, C_2$ such that,
    
    \begin{equation*}
       |\{ f \in H_k^*(N) : p_f > 2\beta \}| \ll k\varphi(N) \exp \left(-C_2 \frac{\beta}{\log \beta} \right),
    \end{equation*}

    provided,

    \begin{equation*}
    C_0 \ll \beta \ll \log(kN).
\end{equation*}
    
   All implied constants are absolute.
\end{corollary}

Assuming $A(k,N) \ll \log kN$, we let $\beta_j = 2^{j-1} A(k,N)$ (and taking $k, N$ sufficiently large) for $j = 0, \dots, m$, so that $\beta_m \gg \log kN$. We split up our sum into dyadic intervals and apply Corollary \ref{cor: large sieve}.

\begin{equation*}
     \sum_{j=0}^{m} \sum_{\substack{ f \in H_k^*(N) \\ 2\beta < p_f \leq 4\beta }} p_f \ll \sum_{j=0}^{m} 2^j A(k,N) \cdot k\varphi(N) \exp \left( -C_2 \frac{2^{j-1}A(k,N)}{\log (2^{j-1}A(k,N))} \right). 
\end{equation*}

We can see this sum is sufficiently small with the following lazy bound for large enough $k, N$,

\begin{equation}
\label{eq: pf 2nd sum}
    \sum_{j=0}^{\infty} 2^j A(k,N) \cdot k\varphi(N) \exp \left( -C_2 \frac{2^{j-1}A(k,N)}{\log (2^{j-1}A(k,N))} \right) \leq k \varphi(N) \sum_{j=0}^{\infty} \frac{1}{2^{j-1} A(k,N)}.
\end{equation}

And we can see this is $o(k\varphi(N))$ as $k+N$ tends to infinity. Then for the case $p_f \gg \log kN$  we may extend the use of Corollary \ref{cor: large sieve} even further. Here we use the GRH bound $p_f \ll (\log k^2N)^2$, where $Q:= k^2N$, to get,

\begin{equation}
\label{eq: pf 3rd sum}
    \smashoperator{\sum_{\substack{ f \in H_k^*(N) \\ \log kN < p_f }}} p_f \ll (\log Q)^2  k\varphi(N) \exp \left(-C_2 \frac{\log kN}{\log \log kN} \right).
\end{equation}

Similarly we see this sum is also $o(k\varphi(N))$ as $k+N$ tends to infinity. So we may use \eqref{eq: pf 2nd sum} and \eqref{eq: pf 3rd sum} to bound the second sum in \eqref{eq:splitsum_pf} as $o(k\varphi(N))$, proving Theorem \ref{thrm:averagepf}.

\subsection{Proof of Theorem \ref{thrm:averagenf}}
\label{Sec:nf proof}
We use that for a fixed prime $p$, $\lambda_f(p)$ is equidistributed in $[-2,2]$ with respect to the $p$-adic Plancherel measure, $\mu_p$ \eqref{eq: p-adic plancherel}. By this we mean the probability that the eigenvalue at $p$ for any form $f \in H_{k_r}^*(N_r)$ lies in a subset $A \subseteq [-2,2]$ tends to $\mu_p(A)$ as $r$ tends to infinity, provided $p \nmid N$.

For this section we explore the Hecke eigenvalues in terms of their associated angles that arise from Deligne's bound, \eqref{eq: thetafp}. We have the following theorem from Serre.

\begin{theorem}[{\cite{Ser}}]

Let $\mathcal{P}$ be a finite set of primes and let $\theta_f = (\theta_f(p))_{p \in \mathcal{P}}$ be an array of angles associated to eigenvalues at all $p \in \mathcal{P}$ for some $f \in H_k^*(N)$. Then the sequence $(\theta_f)_{f \in H_k^*(N)}$ is equidistributed in $[0, \pi] \times ... \times [0, \pi] $ with respect to $\otimes_{p \in \mathcal{P}} \mu_p$, where $\mu_p$ is the $p$-adic Plancherel measure. That is, for any sequence of intervals $(I_p)_{p \in \mathcal{P}}$ indexed by primes, with $I_p \subseteq [0, \pi]$, we have,

   \begin{equation*}
        \lim_{k+N \to \infty} \frac{1}{|H_k^*(N)|} |\{ f \in H_k^*(N): \theta_f(p) \in I_p, \forall \  p \in \mathcal{P}\}| = \prod_{p \in \mathcal{P}} \mu_p(I_p).
    \end{equation*}

     For $k+N$ tending to infinity over even $k$ and over $N$ such that $p \nmid N$ for all $p \in \mathcal{P}$
\end{theorem}

We denote the set on the left hand side as $F_{k,N}(\mathcal{P}, (I_p))$;

\begin{equation*}
    F_{k,N}(\mathcal{P}, (I_p)) := |\{ f \in H_k^*(N): \theta_f(p) \in I_p, \forall \  p \in \mathcal{P}\}|.
\end{equation*}

It is clear that provided $|\mathcal{P}|$ tends to infinity slow enough as $k+N$ tends to infinity, we have,

\begin{equation}
\label{eq: nf main term}
    F_{k,N}(\mathcal{P}, (I_p)) = (1 + o(1))\prod_{p\in \mathcal{P}} \mu_p(I_p) |H_k^*(N)| \textrm{ as } k+N \to \infty.
\end{equation}

However we must also ensure that $p \nmid N$ for all $p \in \mathcal{P}$, and as $\mathcal{P}$ increases to the set of all primes. Hence we must either fix $N=1$ and have $k$ tend to infinity, or have $N$ tend to infinity over integers with large enough prime factors. By proving the latter, we will see that we immediately get the former.

Letting $B(k,N)$ tend to infinity in this manner, and ensuring $N$ has no prime factors less than $B(k,N)$, we then have,

\begin{equation}
\label{eq: nf split}
    \sum_{f \in H_k^*(N)} n_f = \smashoperator{\sum_{\substack{n \leq B(k,N): \\ n=q^a,\  q \textrm{ prime,} \\ a \geq 1}}}n \cdot |\{f\in H_k^*(N): n_f = n\} | + \smashoperator{\sum_{\substack{f \in H_k^*(N): \\ B(k,N) < n_f \leq \log(kN)}}} n_f  \hspace{1em} +\sum_{\substack{f \in H_k^*(N): \\ \log(kN) < n_f}} n_f
\end{equation}

We write $\sum^1$, $\sum^2$ and $\sum^3$ as labels to refer to the first, second or third sum on the right-hand side of \eqref{eq: nf split} respectively.

If we let $\mathcal{P}_N(n) := \{ \textrm{primes } p \leq n: p\nmid N\}$, we have from \eqref{eq: nf main term}, and from the size of $|H_k^*(N)|$ \eqref{eq: size Hk*N}, that,

\begin{align*}
    \mathop{\sum\strut^1} &=  \smashoperator{\sum_{\substack{n \leq B(k,N) \\ n=q^a, \  q \textrm{ prime,} \\ a \geq 1}}} n \cdot F_{k,N}(\mathcal{P}_N(n), (I_{n})) \\
    &= \sum_{\substack{n \geq 1: \\ n=q^a}} n(1+o(1)) \prod_{p \in \mathcal{P}_N(n)} \mu_p(I_{n}) \cdot k\varphi(N) \\
    &= k\varphi(N) \sum_{\substack{n \geq 1: \\ n=q^a } } n\cdot \prod_{p \in \mathcal{P}_N(n)} \mu_p(I_{n}) + o(k\varphi(N)),
\end{align*}

where the sequence $(I_{n})$ corresponds to the intervals that $\theta_f(q)$ must be in such that $n_f = n = q^a$ for any $f \in H_k^*(N)$.

Now we apply the large sieve to $\sum^2$. Similarly to the prime case, we have the following.

\begin{corollary}
\label{cor: nf large sieve}
    For $k\geq 2$ an even integer, and $N$ prime or $N$ only with prime factors larger than $\log kN$, then there exists absolute constants $C_0, C_3 >0$ such that,

    \begin{equation*}
        \# \{ f \in H_k^*(N): n_f > 2\beta \} \ll k\varphi(N) \exp \left( -C_3 \frac{\beta}{\log \beta} \right),
    \end{equation*}
    provided,

    \begin{equation*}
    C_0 \ll \beta \ll \log(kN).
\end{equation*}

    All the implied constants are absolute.
\end{corollary}

Note this corollary requires the level $N$ to be prime or have large enough prime factors, which is not a restriction required in the prime case. As in the prime case, we may assume $B(k,N) \ll \log kN$, and we ensure $N$ tends to infinity over integers with prime factors larger than $\log kN$.

As before, let $\beta_j = 2^{j-1} B(k,N)$ for $j = 0, \dots, m$, so that $\beta_m \gg \log kN$ ($k, N$ sufficiently large), and split $\sum^2$ into dyadic intervals,

\begin{equation*}
    \mathop{\sum\strut^2} \leq \sum_{j=0}^{m}  \sum_{\substack{f \in H_k^*(N): \\ 2\beta_{j} < n_f \leq 4\beta_j } } n_f \ll \sum_{j=0}^{m} k\varphi(N) 2^j B(k,N) \exp \left( -C_3 \frac{2^{j-1}B(k,N)}{\log (2^{j-1}B(k,N))} \right). 
\end{equation*}

As this is exactly the same as in Theorem $\ref{thrm:averagepf}$, see \eqref{eq: pf 2nd sum}, we may say $\sum^2 = o(k\varphi(N))$. The current upper bound on $n_f$ \eqref{eq:nf} is not strong enough to bound $\sum^3$ with the large sieve inequality in Corollary \ref{cor: nf large sieve}. Hence we use Theorem \ref{thrm: linnik nf} where $\varepsilon_p = +1$ for all $p \leq (\log kN)^A$,

\begin{equation*}
        \# \{ f \in H_k^*(N): n_f > (\log kN)^A \} \ll_{\varepsilon_1, \varepsilon_2} N^{1/2 + 1/2A + \varepsilon_1} k^{1/A+  \varepsilon_2}.
    \end{equation*}

Then using Corollary \ref{cor: nf large sieve} for the interval $\log kN < n_f \leq (\log kN)^A$, and Theorem \ref{thrm: linnik nf} for $n_f > (\log kN)^A$, we see, with the current upper bound on the least negative Hecke eigenvalue, $n_f \ll Q^{3/8} = (k^2N)^{3/8}$ \eqref{eq:nf},

\begin{align*}
    \mathop{\sum\strut^3} &= \sum_{\substack{f \in H_k^*(N): \\ \log kN < n_f \leq (\log kN)^A}} n_f + \sum_{\substack{f \in H_k^*(N): \\ n_f > (\log kN)^A}} n_f  \\
    &\ll (\log kN)^A k \varphi(N) \exp \left( -C_4 \frac{\log kN}{\log \log kN} \right) + (k^2N)^{3/8}N^{1/2 + 1/2A+\varepsilon_1}k^{1/A + \varepsilon_2} \\
    &\ll k\varphi(N) + N^{7/8 + 1/2A+\varepsilon_1}k^{3/4 + 1/A+\varepsilon_2}.
\end{align*}

To ensure we have $\sum^3 = o(k \varphi(N))$ we see we need $A>4$. Then with sufficiently large $k$ and $N$ subject to the appropriate conditions we have the desired bound.

It remains to show that the $(I_{n})$ are as in the statement of Theorem \ref{thrm:averagenf}. This can be seen from a special form of Hecke multiplicity, involving Chebyshev polynomials of the second kind, defined by,

\begin{equation}
\label{eq:chebyshev}
    X_n(\theta) := \frac{\sin((n+1)\theta)}{\sin(\theta)}, \hspace{0.5pt} \textrm{ for } \theta \in [0, \pi].
\end{equation}

Then the following can be seen from multiplicity of Hecke eigenvalues,

\begin{equation}
\label{eq:hecke chebyshev mult}
    X_n(\theta_f(p)) = \lambda_f(p^n).
\end{equation}

When we write this as,

\begin{equation*}
    \sin((n+1)\theta_f(p)) = \lambda_f(p^n) \sin (\theta_f(p)),
\end{equation*}

it is obvious that if $\lambda_f(p^j) \geq 0$ for some $j$ we have,

\begin{equation*}
    \theta_f(p) \in A_j =
    \begin{cases}
         [0, \frac{\pi}{j+1}] \cup [\frac{2\pi}{j+1}, \frac{3\pi}{j+1}] \cup ... \cup [\frac{j\pi}{j+1}, \pi] &\text{ if } j \text{  is even,} \\
        [0, \frac{\pi}{j+1}] \cup [\frac{2\pi}{j+1}, \frac{3\pi}{j+1}] \cup ... \cup [\frac{(j-1)\pi}{j+1}, \frac{j\pi}{j+1}]&\text{ if } j \text{ is odd.}
    \end{cases}
\end{equation*}

If $\lambda_f(p^j) \geq 0$ for all $j \in \{ 1, ..., n\}$ for some $n$, then we have,

\begin{equation*}
    \theta_f(p) \in A_n^* = \bigcap_{j=1}^{n} A_j = \left[0, \frac{\pi}{n+1}\right].
\end{equation*}

On the other hand, for a prime $q$, if $\lambda_f(q^n) < 0$, then $\theta_f(q) \in A_n^c$, the complement of $A_n$ in $[0, \pi]$, given by

\begin{equation*}
    A_n^c = \begin{cases}
        (\frac{\pi}{n+1},\frac{2\pi}{n+1}) \cup ... \cup (\frac{(n-1)\pi}{n+1}, \frac{n\pi}{n+1}) &\text{ if } n \text{ is even,} \\
        (\frac{\pi}{n+1},\frac{2\pi}{n+1}) \cup ... \cup (\frac{n\pi}{n+1}, \pi) &\text{ if } n \text{ is odd.}
    \end{cases}
\end{equation*}

If $\lambda_f(q^n)$ is the first prime power of $q$ with negative eigenvalue, we have,

\begin{equation*}
    \theta_f(q) \in B_n^* = \left( \bigcap_{j=1}^{n-1}A_j \right) \cap A_n^c = \left(\frac{\pi}{n+1} , \frac{\pi}{n}\right).
\end{equation*}

Thus, if for some newform $f \in H_k^*(N)$ we have $n_f = q^n$ for some prime $q$ and $n \geq 1$, then we must have that for all primes $p < q^n$ but not including $q$, that

\begin{equation*}
    \theta_f(p) \in A_{a_p(q^n)}^*,  
\end{equation*}

where $a_p(q^n) = \lfloor n \log q/ \log p \rfloor$, the greatest power such that $p^{a_p(q^n)} < q^n$. At $q$ we have,

\begin{equation*}
    \theta_f(q) \in B_n^*.
\end{equation*}

These intervals are exactly those stated in Theorem \ref{thrm:averagenf}, where for $B_n^*$ we have used the closed interval, as $\mu_p((\alpha, \beta)) = \mu_p([\alpha, \beta])$.

\section{Proof of auxiliary results}

\label{sec:aux proof}
In this section we prove auxiliary results pertaining to the proofs of the main theorems.

\subsection{Proof of Lemma \ref{lem:pf half}}
\label{subsec:proof equidist}

We are naturally required to consider the cases of primes that divide the level and primes that do not. The structure of eigenvalues at primes of both types is well-known, and outlined in Section \ref{sec:motivation}. Indeed the distribution of eigenvalues at $p\nmid N$ is also well-known, via the Sato-Tate conjecture at the global scale \cite{BarGerHarTay} and by the $p$-adic Plancherel measure \eqref{eq: p-adic plancherel} at the local scale \cite{Ser}. In either case, they are equally distributed between positive and negative signs. We will see that the same applies for eigenvalues at primes that divide the level exactly.

We will refer to eigenvalues at primes that divide the level as `Atkin-Lehner eigenvalues' due to the development of the theory of newforms by Atkin and Lehner \cite{AtkLeh}, and its generalisations by Atkin and Li \cite{AtkLi}. We repeat the structure of Atkin-Lehner eigenvalues below.

\begin{equation}
\label{eq:A-L e values}
    \lambda_f(p) = \begin{cases}
        0 &\textrm{ if } p^2 |N, \\
        \pm p^{-1/2} &\textrm{ if } p|N \textrm{ but } p^2 \nmid N.
    \end{cases}
\end{equation}

Let $\mathcal{P}$ denote a finite set of primes. We can see the statement of Lemma \ref{lem:pf half} is equivalent to showing,

\begin{equation}
\label{eq:equidist char fct}
    \lim_{k+N \to \infty}\frac{1}{|H_k^*(N)|} \sum_{f\in H_k^*(N)} \prod_{p \in \mathcal{P}}\mathbbm{1}_{[0,2]}(\lambda_f(p)) = \prod_{p \in \mathcal{P}} \frac{1}{2},
\end{equation}

where $\mathbbm{1}_{[0,2]}$ is the indicator function of the set $[0,2]$. By approximating the indicator function with continuous functions, it is sufficient to show, for all real, continuous functions $g_p$ indexed by primes $p\in \mathcal{P}$, defined on $[-2,2]$, that

\begin{equation}
\label{eq:equidist}
    \lim_{k+N \to \infty}\frac{1}{|H_k^*(N)|} \sum_{f\in H_k^*(N)} \prod_{p \in \mathcal{P}}g_p(\lambda_f(p)) - \prod_{p \in \mathcal{P}} \langle g_p, \rho_{p,N} \rangle_{[-2,2]} = 0,
\end{equation}

with $\rho_{p,N}$ the measure describing the distribution of $\lambda_f(p)$ in $[-2,2]$ over $f\in H_k^*(N)$, and $\langle g_p,\rho_{p,N} \rangle_{[-2,2]}$ defined by,

\begin{equation*}
    \langle g_p, \rho_{p,N} \rangle_{[-2,2]} = \langle g_p, \rho_{p,N} \rangle = \int_{-2}^2 g_p(x) \rho_{p,N}(x).
\end{equation*}

The measures $\rho_{p,N}$ depend on the divisibility of $N$ by $p$, so therefore they depend on $N$, and we will see that $\langle 1, \rho_{p,N} \rangle_{[0,2]} = 1/2$ for all $p \in \mathcal{P}$ with the appropriate measures. The case where $p\nmid N$ for all $p \in \mathcal{P}$, with $\rho_{p,N}=\mu_p$ the $p$-adic Plancherel measure, has been shown by Serre \cite[Theorems 2,3]{Ser}.

Let $\mathcal{P}=\mathcal{P}_1 \cup \mathcal{P}_2$ be any disjoint union of primes in $\mathcal{P}$, so $\mathcal{P}_1 \cap \mathcal{P}_2 = \emptyset$, and set $N = \left(\prod_{p \in \mathcal{P}_2} p\right) \cdot N'$, such that $(N',p)=1$ for all $p \in \mathcal{P}_1$, and so that $N$ is square-free. As mentioned above, it remains to show the case when $\mathcal{P}_2 \neq \emptyset$.

We consider \eqref{eq:equidist} in terms of traces of Hecke operators. That is, we let $S_k(N)$ denote the space of cuspforms of weight $k$ and level $N$. Then $S_k^{\text{new}}(N)$ denotes the space of primitive cuspforms, so that $H_k^*(N)$ is a set of newforms that generate the set $S_k^{\text{new}}(N)$. For each $n \geq 1$ we have $T_n(N,k)$ the $n^{\text{th}}$ Hecke operator acting on $S_k(N)$. Similarly, $T_n^{\text{new}}(N,k)$ is the $n^{\text{th}}$ Hecke operator acting on $S_k^{\text{new}}(N)$ (and by definition also on $H_k^*(N)$).

Then the trace of $T_n^{\text{new}}(N,k)$ returns the sum of the eigenvalues at $n$,

\begin{equation*}
    n^{-\frac{k-1}{2}}\mathrm{Tr}T_n^{\text{new}}(N,k) = \sum_{f \in H_k^*(N)} \lambda_f(n).
\end{equation*}

For brevity, we write $T_n^*$ to mean $T_n^{\text{new}}(N,k)$ normalised by $n^{(k-1)/2}$. By the linearity of the trace, for some degree $m$ polynomial $P$ with real coefficients the trace of $P(T_n^*)$ is equal to the sum over $f$ of $P(\lambda_f(n))$. Then we may rewrite \eqref{eq:equidist} as,

\begin{equation}
\label{eq:trace equi}
    \lim_{k+N' \to \infty} \frac{1}{|H_k^*(N)|} \mathrm{Tr}\Big( \prod_{p \in \mathcal{P}}P_p(T_p^*)\Big) = \prod_{p \in \mathcal{P}} \langle P_p, \rho_p \rangle,
\end{equation}

for polynomials $P_p$ indexed by primes $p \in \mathcal{P}$. We note that we are taking the limit in $k$ and $N'$, so as to fix the primes divisible by the level, and thus $\rho_p$ is no longer dependent on $N$. As polynomials are dense in the space of continuous functions, to prove \eqref{eq:equidist char fct} it suffices to show \eqref{eq:trace equi} for any polynomial of degree $m$, for all $m \geq 0$, and for any arbitrary partition $\mathcal{P}_2$ of $\mathcal{P}$. Our measures are as follows,

\begin{equation*}
    \rho_p = \begin{cases}
        \mu_p &\text{ for } p \in \mathcal{P}_1, \\
        \frac{1}{2}\delta_{\Omega_p} &\text{ for } p \in \mathcal{P}_2, \  \Omega_p = \{-p^{-1/2}, p^{-1/2}\}
    \end{cases}
\end{equation*}

where $\mu_p$ is the $p$-adic Plancherel measure, and $\delta$ is the Dirac measure for the set $\Omega_p = \{ -p^{-1/2}, p^{-1/2}\}$. For the polynomials we let $m_p \geq 0$ be integers indexed by $p \in \mathcal{P}$ and let,

\begin{equation}
\label{eq:polynomial}
    P_{m_p}(T_p^*) = \begin{cases}
        X_{m_p}(T_p^*) &\text{ for } p \in \mathcal{P}_1, \\
        (T_p^*)^{m_p} &\text{ for } p \in \mathcal{P}_2,
    \end{cases}
\end{equation}

where $X_m$ are Chebyshev polynomials of the second kind \eqref{eq:chebyshev}. By the Hecke relations we have for $p \in \mathcal{P}_1$,

\begin{equation}
\label{eq:X_mT_p}
    X_m(T_p^*) = T_{p^m}^*.
\end{equation}

Similarly, by the total multiplicativity of the Atkin-Lehner eigenvalues, when $p \in \mathcal{P}_2$ we have,

\begin{equation}
\label{eq:ALT_p}
    (T_p^*)^{m} = T_{p^m}^*.
\end{equation}

For the right hand side of \eqref{eq:trace equi}, we have (e.g. \cite{Ser}) for $p \in \mathcal{P}_1$,

\begin{equation}
\label{eq:X_m mu_p}
    \langle X_m, \mu_p \rangle = \begin{cases}
        p^{-m/2} &\text{ if } m \text{ is even}, \\
        0 &\text{ if } m \text{ is odd}.
    \end{cases}
\end{equation}

Furthermore, for $p \in \mathcal{P}_2$, we have, 

\begin{equation}
\label{eq:x^m 1/2}
    \left\langle x^m , \frac{1}{2}\delta_{\Omega_p} \right\rangle = \frac{1}{2} ( (-p^{-1/2})^m + (p^{-1/2})^m ) = \begin{cases}
        p^{-m/2} &\text{ if } m \text{ is even}, \\
        0 &\text{ if } m \text{ is odd},
    \end{cases}
\end{equation}

Let  $n = \prod_{p\in\mathcal{P}} p^{m_p}$. Using the polynomials \eqref{eq:polynomial}, we have by \eqref{eq:X_mT_p}, \eqref{eq:ALT_p}, \eqref{eq:X_m mu_p} and \eqref{eq:x^m 1/2} that we are left to show, for any $m_p \geq 0$,

\begin{equation}
\label{eq: tr square}
    \lim_{k+N' \to \infty} \frac{1}{|H_k^*(N)|} \mathrm{Tr}(T_n^*) = \begin{cases}
        n^{-1/2} &\text{ if } n \text{ is a square}, \\
        0 &\text{ otherwise},
    \end{cases}
\end{equation}

via multiplicativity of the Hecke operators. We briefly sketch the case $\mathcal{P}_2 = \emptyset$, shown by Serre, where $N'=N$. We use the following proposition:

\begin{proposition}[{\cite{Ser}}]
\label{prop:serre}
    If $(n,N) =1$ then we have for all $\varepsilon > 0$,

    \begin{equation*}
    \left| \mathrm{Tr}(T_n^*) - \frac{k-1}{12}\chi(\sqrt{n})n^{-1/2}\psi^*(N) \right| \ll_n N^{1/2+2\varepsilon},
    \end{equation*}
    where $\chi$ is the trivial Dirichlet character modulo $N$. We have that $\psi^*(N)$ is multiplicative and defined on primes as follows,

    \begin{equation*}
        \psi^*(p^{\alpha}) = \begin{cases}
            1 &\text{ if } \alpha = 0, \\
            p-1 &\text{ if } \alpha = 1, \\
            p^2 - p - 1 &\text{ if } \alpha = 2, \\
            p^{\alpha} - p^{\alpha- 1} - p^{\alpha - 2} + p^{\alpha -3} &\text{ if } \alpha > 2.
        \end{cases}
    \end{equation*}
\end{proposition}

First, we can see that for square-free $N$ we have that $\psi^*(N) = \varphi(N)$ the Euler totient function. Then for this proof we may replace $\psi^*(N)$ in Proposition \ref{prop:serre} with $\varphi(N)$.

Following the proof \cite[Theorem 2]{Ser}, we have from the definition of the trace that $\mathrm{Tr}(T_1^*) = |H_k^*(N)|$, thus by Proposition \ref{prop:serre} we can see that $|H_k^*(N)| \sim \frac{k-1}{12}\varphi(N)$. Then dividing the left hand side of Proposition \ref{prop:serre} by $\frac{k-1}{12}\varphi(N)$ we see the first term is now equivalent to the left hand side of \eqref{eq: tr square}. Further, we can see that dividing the right hand side of Proposition \ref{prop:serre} by $\frac{k-1}{12}\varphi(N)$ implies that the right hand side will tend to 0 as $k+N$ tends to infinity, proving \eqref{eq: tr square} provided $N$ tends to infinity with $(n,N) = 1$.

We show \eqref{eq: tr square} holds for $(n,N)>1$ provided $N$ is square-free via Proposition \ref{prop:serre}. Let $\prod_{p \in \mathcal{P}_1} p^{m_p} = P$ and $\prod_{p \in \mathcal{P}_2} p^{m_p} = Q$ so that $n=PQ$, and we have that $(N,P)=1$ but $(N,Q)>1$. First we note that by total multiplicativaty of Atkin-Lehner eigenvalues, and \eqref{eq:A-L e values}, for some integer $m \geq 0$ (and $p|N$),

\begin{equation}
\label{eq: trace square AL}
    \mathrm{Tr}(T_{p^{2m}}^*) = \sum_{f \in H_k^*(N)} \lambda_f(p^{2m}) = \sum_{f \in H_k^*(N)} \lambda_f(p)^{2m} = \frac{|H_k^*(N)|}{p^m}.
\end{equation}

Suppose $Q$ is a square. Then $m_p$ is even for each $p \in \mathcal{P}_2$. We have,

\begin{equation*}
    \mathrm{Tr}(T_n^*) = \mathrm{Tr}(T_P^*T_Q^*)= \sum_{f \in H_k^*(N)} \lambda_f(P)\prod_{p\in\mathcal{P}_2} \lambda_f(p)^{m_p}= \prod_{p \in \mathcal{P}_2} p^{-\frac{m_p}{2}} \mathrm{Tr}(T_P^*).
\end{equation*}

And we may appeal to Proposition \ref{prop:serre}, in the limit for $N'$. Now suppose $Q$ is not a square. We may similarly factor out even prime powers in the factorisation of $Q$, leaving us only with the case that $Q$ is square-free. In this case we aim to show $Tr(T_n^*) = 0$ in the limit as $N'$ tends to infinity. We approach this by showing Proposition \ref{prop:serre} holds for $(n,N)>1$ where $n$ is not a square.

Indeed, the proof of Proposition \ref{prop:serre} in \cite{Ser} is achieved by separating $\mathrm{Tr}(T_n^{\textrm{new}}(N,k))$ into a weighted sum of $\mathrm{Tr}(T_n(M,k))$ for divisors $M$ of $N$, via the newform trace formula. From there, the proof is finished by computing the traces of $T_n(M,k)$ with the Eichler-Selberg trace formula, which we note is valid on all $n$, including $(n,M)>1$ (see e.g. \cite{SchVlu}). Then here we show the newform trace formula for $(n,N)=Q>1$ (that is, $Q$ square-free) is the same as that when $(n,N)=1$.

For trivial character and square-free level we have, following the notation of \cite{CohStr},

\begin{equation*}
    \mathrm{Tr}(T_n^{\text{new}}(N,k)) = \sum_{M|N} \sum_{\substack{ d|M \\ d^2|n}} d^{k-1} \beta_{n/d^2} \left( \frac{N}{M} \right) \mathrm{Tr}\left(T_{n/d^2}\left(\frac{M}{d}, k\right)\right),
\end{equation*}

where for $m \geq 1$ we have that $\beta_m(n)$ is defined on prime powers:

\begin{equation*}
    \beta_m(p^{\alpha}) = \begin{cases}
        -2 &\text{ if } p \nmid m, \alpha = 1, \\
        1 &\text{ if } p \nmid m, \alpha = 2, \\
        0 &\text{ if } p \nmid m, \alpha \geq 3, \\
        \mu(p^{\alpha}) &\text{ if } p | m,
    \end{cases}
\end{equation*}

and $\mu$ is the M\"obius function. For the inner sum, recall that when $Q$ is square-free we have $(N,n)=Q$ so for a divisor $M$ of $N$ we have that if $d|M$ and $d^2|n$ then we must have $d^2|Q$. However, as we have that that $Q$ is square-free, the inner sum only consists of the $d=1$ term. Thus,

\begin{equation*}
    \mathrm{Tr}(T_n^{\text{new}}(N,k)) = \sum_{M|N} \beta_{n}\left( \frac{N}{M} \right) Tr(T_{n}(M,k)).
\end{equation*}

This is the same formula as when $(n,N) = 1$, and hence we see from \cite[Theorem 2]{Ser}, where the $\mathrm{Tr}(T_n(M,k))$ are bounded via the Eichler-Selberg trace formula, that Proposition \ref{prop:serre} holds for $(n,N)>1$, with $N$ square-free, $n=PQ$ and $(N,n)=Q$. Then \eqref{eq: tr square} holds for $N =  \left(\prod_{p \in \mathcal{P}_2} p\right) \cdot N'$ for any arbitrary partition $\mathcal{P}_2$ of $\mathcal{P}$. Finally we can see $\langle 1, \rho_p \rangle_{[0,2]} = 1/2$, as required.

\subsection{Proof of Corollaries \ref{cor: large sieve} and \ref{cor: nf large sieve}}

These corollaries are consequences of the following large sieve inequality, which itself is a modified version of the large sieve inequality appearing in \cite[Theorem 2]{klsw}.

\begin{proposition}
\label{prop: large sieve}
Let $k \geq 2$ and $N$ be square-free. Fix $\nu \geq 1$ an integer. Let $\mathscr{P}$ be a set of prime numbers of positive density in the following sense:

\begin{equation}
\label{eq: positive prime density}
    \sum_{\substack{s<p\leq 2s \\ p \in \mathscr{P}}} \frac{1}{p} \geq \frac{\delta}{\log s}, \hspace{2cm} (s \geq s_0)
\end{equation}

for some constants $\delta > 0$ and $s_0>0$. Let $(\varepsilon_p)_{p \in \mathscr{P}}$ be a sequence of signs indexed by primes in $\mathscr{P}$. Then there exists two positive constants $C_0$ and $C_1$ such that

\begin{equation*}
   |\{ f \in H_k^*(N): \varepsilon_p\lambda_f(p^{\nu})>0 \textrm{ for } p \in \mathscr{P},\  \beta <p\leq 2 \beta\}| \ll_{\mathscr{P}, \nu} k\varphi(N) \exp \left(-C_1 \frac{\beta}{\log \beta}\right),
\end{equation*}

provided,

\begin{equation*}
    C_0 \ll_{\nu, \mathscr{P}} \beta \ll_{\nu, \mathscr{P}} \log(kN).
\end{equation*}

The constants $C_0, C_1$ and all implied constants depend only on $\nu$ and $\mathscr{P}$.
\end{proposition}

Note that we do not need the requirement of $p$ coprime to $N$.

When $\mathscr{P}$ is the set of all primes, $\nu=1$ and $\varepsilon_p = 1$ for all primes $p$, we see that Proposition \ref{prop: large sieve} provides an upper bound on the set of newforms with $p_f$ greater than $2\beta$, hence we have Corollary \ref{cor: large sieve}. Similarly, we may get Corollary \ref{cor: nf large sieve} although we cannot include eigenvalues at primes that divide the level. We will see that this restriction requires us to have $N$ be prime, or only have prime factors larger than $\log kN$.

This proof follows the same structure as its adapted counterpart in \cite[Theorem 2]{klsw}. The intuition is that if $\varepsilon_p \lambda_f(p^{\nu}) \geq 0$ for all primes in an interval, then its sum over that interval will be large. Via the large sieve, we can show that this is unlikely to happen. However we must take into account that if the $\lambda_f(p^{\nu})$ are all very small in absolute value in this interval, then the sum over $p$ of $\varepsilon_p \lambda_f(p)$ will be small in this interval regardless of its signs. The large sieve can show that this is also unlikely to happen.

\begin{lemma}{\cite[Theorem 1]{LauWu}}
\label{lem: large sieve modular forms}
Let $\nu \geq 1$ be a fixed integer and let $\{b_p\}_p$ be a sequence of real numbers indexed by prime numbers such that $|b_p| \leq B$ for some constant $B$ and for all primes $p$. Then,

\begin{equation*}
    \sum_{f \in H_k^*(N)} \Bigg| \sum_{\substack{P<p \leq Q \\ p\nmid N}} b_p \frac{\lambda_f(p^{\nu})}{p} \Bigg|^{2j} \ll k\varphi(N) \left( \frac{96B^2(\nu + 1)^2j}{P\log P} \right)^j + (kN)^{10/11}\left(\frac{10BQ^{\nu/10}}{\log P} \right)^{2j}
\end{equation*}

uniformly for the constant $B$, integers $j$, $k$, $P$, $Q$ and $N$ such that,

\begin{equation*}
    B>0, \quad j \geq 1, \quad 2|k, \quad 2 \leq P <Q \leq 2P, \quad N \geq 1 \   (squarefree).
\end{equation*}

The implied constant depends only on $\nu$.
\end{lemma}

For Corollary \ref{cor: nf large sieve} we may simply use Lemma \ref{lem: large sieve modular forms} as it is, but for Corollary \ref{cor: large sieve} we must include primes that divide $N$ to the sieved sum. Hence we need only prove Corollary \ref{cor: large sieve}, and Corollary \ref{cor: nf large sieve} will follow. With the notation as above we consider,

\begin{align}
   \sum_{f \in H_k^*(N)} \left| \sum_{P < p \leq Q} b_p \frac{\lambda_f(p^{\nu})}{p} \right|^{2j} &= \sum_{f \in H_k^*(N)} \Bigg| \sum_{\substack{P < p \leq Q \\ p\nmid N}} b_p \frac{\lambda_f(p^{\nu})}{p} +\sum_{\substack{P < p \leq Q \\ p| N}} b_p \frac{\lambda_f(p^{\nu})}{p} \Bigg|^{2j} \nonumber \\
   &\leq \sum_{f \in H_k^*(N)} 2j \cdot 2^{2j} \Bigg| \sum_{\substack{ P < p \leq Q \\ p\nmid N}}   b_p \frac{\lambda_f(p^{\nu})}{p} \Bigg|^{2j} + 2j\cdot 2^{2j}\Bigg|\sum_{\substack{P < p \leq Q \\ p | N}}  b_p \frac{\lambda_f(p^{\nu})}{p} \Bigg|^{2j}. \label{eq:p|N, p nmid N}
\end{align}

We may simplify the second inner sum, via Atkin-Lehner-Li theory \eqref{eq:ALL theory}, as a sum over reciprocals of prime powers, where the powers are strictly greater than 1. As such, this sum is clearly small, and indeed by the Prime Number Theorem one can see,

\begin{equation*}
    \Big| \sum_{\substack{P<p \leq Q \\ p|N}}  b_p \frac{\lambda_f(p^{\nu})}{p} \Big| \leq B \sum_{\substack{P<p\leq Q \\ p|N}} \frac{1}{p^{(\nu+2)/2}} =  \frac{B}{P^{\nu/2} \log P}(1+o(1)).
\end{equation*}

From here, with Lemma \ref{lem: large sieve modular forms}, the right hand side of \eqref{eq:p|N, p nmid N} has the following contribution, where we have trivially bounded the factor of $j$.

\begin{align}
    &\ll k\varphi(N) j\left( \frac{384B^2 (\nu+1)^2j}{P \log P}\right)^j + k\varphi(N)j\left(\frac{4B^2}{P^{\nu}\log^2 P} \right)^j + (kN)^{10/11}j\left( \frac{20BQ^{\nu/10}}{\log P} \right)^{2j} \nonumber \\
    &\ll  k\varphi(N) \left( \frac{768B^2 (\nu+1)^2j}{P \log P}\right)^j + (kN)^{10/11}\left( \frac{40BQ^{\nu/10}}{\log P} \right)^{2j}.
    \label{eq:p|N sieve}
\end{align}

Now we define the following sets. Let $\mathscr{P}$, $\delta$ be as in Proposition \ref{prop: large sieve}, and let

\begin{align*}
	\mathscr{E}_{\beta}(k, N;\mathscr{P}) = \mathscr{E}_{\beta} &:= \{ f \in H_k^*(N) : \   \varepsilon_p \lambda_f(p^{\nu}) > 0 \textrm{ for } p \in \mathscr{P} \cap (\beta,2\beta] \}, \hspace{1em} \beta \geq 2, \\
	\mathscr{E}_{\beta}^{\nu'}(k,N;\mathscr{P}) = \mathscr{E}_{\beta}^{\nu'} &:= \left\{ f \in H_k^*(N) : \  \Bigg| \sum_{\substack{\beta<p\leq 2\beta \\ p \in \mathscr{P}}} \frac{\lambda_f(p^{2\nu'})}{p}\Bigg| \geq \frac{\delta}{2\nu \log \beta} \right\}, \hspace{2em} (1 \leq \nu' \leq \nu).
\end{align*}

We wish to bound the size of the set $\mathscr{E}_{\beta}$ where the exact dependence of $\beta$ is yet to be determined. With the definition of $\mathscr{E}_{\beta}$ and Deligne's inequality we have,

\begin{align*}
    \sum_{f \in \mathscr{E}_{\beta}} \Bigg| \sum_{\substack{\beta<p\leq 2\beta \\ p \in \mathscr{P}}} \frac{\lambda_f(p^{\nu})^2}{p} \Bigg|^{2j} &\leq \sum_{f \in \mathscr{E}_{\beta}} \Bigg| \sum_{\substack{\beta<p\leq 2\beta \\ p \in \mathscr{P}}} (\nu+1)\varepsilon_p \frac{\lambda_f(p^{\nu})}{p} \Bigg|^{2j} \\
    &\leq  \sum_{f \in H_k^*(N)} \Bigg| \sum_{\substack{\beta<p\leq 2\beta \\ p \in \mathscr{P}}} (\nu+1)\varepsilon_p \frac{\lambda_f(p^{\nu})}{p} \Bigg|^{2j}.
\end{align*}

We apply the large sieve \eqref{eq:p|N sieve} with,

\begin{equation*}
    b_p = \begin{cases}
        (\nu+1)\varepsilon_p &\textrm{ for } p \in \mathscr{P}, \\
        0 &\textrm{ otherwise},
    \end{cases}
\end{equation*}

to find,

\begin{equation}
\label{eq: E_k^* large sieve}
	\sum_{f \in \mathscr{E}_{\beta}} \Bigg| \sum_{\substack{\beta<p\leq 2\beta \\ p \in \mathscr{P}}} \frac{\lambda_f(p^{\nu})^2}{p} \Bigg|^{2j} \ll k\varphi(N) \left( \frac{768(\nu+1)^4j}{\beta\log \beta}\right)^j + (kN)^{10/11}\left( \frac{40(\nu+1)(2\beta)^{\nu/10}}{\log \beta} \right)^{2j}.
\end{equation}

We then use the relation \eqref{eq:Heckemult}, the positive density of the set $\mathscr{P}$ and the definition of $\mathscr{E}_{\beta}^{\nu'}$ to trivially bound the left-hand side,

\begin{align}
\label{eq:prime logkn rough}
    &\geq \sum_{f \in \mathscr{E}_{\beta} \backslash (\cup_{\nu' =1}^{\nu} \mathscr{E}_{\beta}^{\nu'}) } \left( \sum_{\substack{\beta < p \leq 2\beta \\ p \in \mathscr{P}}} \frac{1}{p} - \sum_{1 \leq \nu' \leq \nu} \Bigg|\sum_{\substack{ \beta < p \leq 2\beta \\ p \in \mathscr{P}}} \frac{\lambda_f(p^{2\nu'})}{p} \Bigg| \right)^{2j} \\
	&\geq \sum_{f \in \mathscr{E}_{\beta}\backslash (\cup_{\nu' =1}^{\nu} \mathscr{E}_{\beta}^{\nu'})} \left( \frac{\delta}{2\log \beta} \right)^{2j}. \label{eq: E_beta lower bound}
\end{align}

In the case of Corollary \ref{cor: nf large sieve} we must restrict $\mathscr{P}$ to $\mathscr{P}_N$, the set of primes in $\mathscr{P}$ that do not divide $N$. In that case,

\begin{equation*}
    \sum_{f \in \mathscr{E}_{\beta}} \Bigg| \sum_{\substack{\beta < p \leq 2 \beta \\ p \in \mathscr{P}_N}} \frac{\lambda_f(p^{\nu})^2}{p} \Bigg|^{2j} \geq  \sum_{f \in \mathscr{E}_{\beta} \backslash (\cup_{\nu' =1}^{\nu} \mathscr{E}_{\beta}^{\nu'}) } \left( \sum_{\substack{\beta < p \leq 2\beta \\ p \in \mathscr{P}}} \frac{1}{p} - \sum_{\substack{\beta < p \leq 2\beta \\ p|N}} \frac{1}{p} - \sum_{1 \leq \nu' \leq \nu} \Bigg|\sum_{\substack{ \beta < p \leq 2\beta \\ p \in \mathscr{P}}} \frac{\lambda_f(p^{2\nu'})}{p} \Bigg| \right)^{2j}
\end{equation*}

And we have the same lower bound as in \eqref{eq: E_beta lower bound} if the set of primes between $\beta$ and $2\beta$ that divide $N$ is empty, thus we require $N$ to be prime, or have no prime factors smaller than $2\beta$. Then with \eqref{eq: E_k^* large sieve},  

\begin{equation}
 \label{eq: Ek*-Ek'}
    |\mathscr{E}_{\beta} \backslash \cup \mathscr{E}_{\beta}^{\nu'}| \ll k\varphi(N) \left( \frac{3072(\nu +1)^4 j \log \beta}{\beta \delta^2} \right)^j + (kN)^{10/11}\left( \frac{80(\nu+1)(2\beta)^{\nu/10}}{\delta } \ \right)^{2j}.
\end{equation}

This leaves us to bound $\mathscr{E}_{\beta}^{\nu'}$ for each $\nu'$. In the large sieve \eqref{eq:p|N sieve} we take,

\begin{equation*}
    \nu =2\nu', \quad B=1, \quad P=\beta, \quad Q=2\beta, \quad \textrm{ and } b_p = \begin{cases}
        1 &\textrm{ if } p \in \mathscr{P}, \\
        0 &\textrm{ otherwise}.
    \end{cases}
\end{equation*}

We then use the definition of $\mathscr{E}_{\beta}^{\nu'}$,

\begin{align*}
    |\mathscr{E}_{\beta}^{\nu'}| \left( \frac{\delta}{2\nu\log \beta} \right)^{2j} \leq \sum_{f \in \mathscr{E}_{\beta}^{\nu'}} \Bigg| \sum_{\substack{\beta<p\leq 2\beta \\ p \in \mathscr{P}}} &b_p \frac{\lambda_f(p^{2\nu'})}{p} \Bigg|^{2j} \leq  \sum_{f \in H_k^*(N)} \Bigg| \sum_{\substack{\beta<p\leq 2\beta \\ p \in \mathscr{P}}} b_p\frac{\lambda_f(p^{2\nu'})}{p} \Bigg|^{2j} \\
    &\ll k\varphi(N) \left( \frac{768(2\nu'+1)^2j}{\beta \log \beta} \right)^j + (kN)^{10/11} \left( \frac{40(2\beta)^{\nu'/5}}{\log \beta} \right)^{2j}.
\end{align*}

We can see that $\nu^2(2\nu'+1)^2 \leq 4(\nu+1)^4$ for $1\leq \nu' \leq \nu$. Thus we have,

\begin{equation*}
    |\mathscr{E}_{\beta}^{\nu'}| \ll k\varphi(N) \left( \frac{12288(\nu+1)^4 j \log \beta}{\delta^2 \beta} \right)^j + (kN)^{10/11} \left( \frac{80\nu(2\beta)^{\nu/5}}{\delta} \right)^{2j}.
\end{equation*}

Together with \eqref{eq: Ek*-Ek'}, we have,

\begin{equation}
\label{eq: E_k^*}
    |\mathscr{E}_{\beta}(k,N;\mathscr{P})| \ll k\varphi(N) \left( \frac{12288(\nu+1)^4 j\log \beta}{\delta^2\beta} \right)^j + (kN)^{10/11}\left( \frac{80(\nu+1)(2\beta)^{\nu/5}}{\delta} \right)^{2j}.
\end{equation}

It now remains to pick a suitable integer $j \geq 1$ and determine the dependence of $\beta$ on $k$ and $N$. Consider, 

\begin{equation*}
    j = \left\lfloor \delta^*\frac{\beta}{\log \beta} \right\rfloor,
\end{equation*}

where

\begin{equation*}
    \delta^* = \frac{\delta^2}{15000(\nu+1)^4}.
\end{equation*}

To ensure $j \geq 1$ we choose $\beta$ large enough such that $\beta/\log \beta \geq 1/\delta^*$, and hence we have some constant $C_0(\nu, \mathscr{P}) = C_0$ such that we require $C_0 \ll_{\nu, \mathscr{P}} \beta$. Our main term becomes, for some constant $c > 1$,

\begin{align*}
    \left( \frac{12288(\nu+1)^4 j\log \beta}{\delta^2\beta} \right)^j &\leq  \left( \frac{1}{c} \right)^{j}\\
    &\ll_{\nu, \mathscr{P}} \exp \left( -C_1 \frac{\beta}{\log \beta} \right).
\end{align*}

For large enough $\beta$, and where $C_1$ is a positive constant depending on $\nu$, $\mathscr{P}$. For the error term in \eqref{eq: E_k^*} we require $ \ll (kN)^{\alpha}$ for some $\alpha(\nu, \mathscr{P}) = \alpha$ sufficiently small and depending on $\nu$ and $\mathscr{P}$. That is,

\begin{equation*}
 j\log \beta \ll \alpha \log(kN).
\end{equation*}

Hence we have $\beta \ll_{\nu, \mathscr{P}} \log(kN)$ and large enough $k, N$.

\section{Proof of weighted large sieve and statistical results}
\label{sec:proof of large sieve}

Here we will prove Theorems \ref{thrm:weighted linnik sieve} and \ref{thrm: linnik nf}. We will see that we may view Theorem \ref{thrm: linnik nf} as more of a corollary of Theorem \ref{thrm:weighted linnik sieve}, except that we use it in the form of Corollary \ref{cor:weighted polynomial sieve}, stated below. Although seemingly backwards, we will begin by assuming Theorem \ref{thrm:weighted linnik sieve} and Corollary \ref{cor:weighted polynomial sieve} so that we may prove Theorem \ref{thrm: linnik nf}. This will allow us to easily introduce Corollary \ref{cor:weighted polynomial sieve}, which itself can be viewed as a large sieve for an approximation of sets. Then we will finish by proving Theorem \ref{thrm:weighted linnik sieve} and Corollary \ref{cor:weighted polynomial sieve}.

\subsection{Proof of Theorem \ref{thrm: linnik nf}}

Our goal is to sieve sets of the type,

\begin{equation*}
    S = \{ f \in H_k^*(N): \theta_f(p) \in I_p, \  p \leq (\log kN)^A, p \nmid N\},
\end{equation*}

for some intervals $I_p \subseteq [0, \pi]$ that represent $\lambda_f(p)\varepsilon_p > 0$, and some $A > 1$.  To use our sieve we approximate the characteristic function of $I_p$ with Chebyshev polynomials, since they form an orthonormal basis in $L^2([0,\pi],\mu_{ST})$. However, such an approximation would involve infinitely many terms, hence we consider a finite approximation. That is, let $Y_p$ be a degree $s$ Chebyshev polynomial, indexed by some prime $p$,

     \begin{equation}
     \label{eq:chebyshev polynomial}
        Y_p(\theta) = \sum_{i =0}^{s} \alpha_p(i) X_i(\theta),
     \end{equation}

    and where $\alpha_p(i)$ are real coefficients of $Y_p$. Recall $X_n$ is the $n^{\text{th}}$ Chebyshev function \eqref{eq:chebyshev}. Note that by orthonormality we have,

    \begin{equation*}
        \int_0^{\pi} Y_p(\theta) \text{ d}\mu_{ST} = \alpha_p(0),
    \end{equation*}

    hence when constructing our approximation to the characteristic function of $I_p$, we can consider functions where $Y_p(\theta_f(p))$ is close to the mean of $Y_p$. In other words,

    \begin{equation*}
        \theta_f(p) \in I_p \implies |Y_p(\theta_f(p)) -\alpha_p(0)| \leq \delta,
    \end{equation*}

    for some $\delta > 0$. On the other hand, we may equivalently consider functions that have $Y_p(\theta_f(p))$ far from its mean:

    \begin{equation*}
        \theta_f(p) \notin I_p \implies Y_p(\theta_f(p)) - \alpha_p(0) \geq \delta \hspace{0.5em} \textrm{OR} \hspace{0.5em} Y_p(\theta_f(p)) - \alpha_p(0) \leq -\delta,
    \end{equation*}

    for some $\delta > 0$. Then we may make an approximation on the size of $S$ that is somewhat easier to form by just considering one of the conditions that imply $Y(\theta_f(p))$ is far from its mean, notably when $Y_p(\theta_f(p)) \leq \alpha_p(0) - \delta$. In this vein we have,

    \begin{corollary}
    \label{cor:weighted polynomial sieve}
    Let $Y_p$ be a Chebyshev polynomial of degree $\leq s$, and indexed by primes $p \leq \beta$ that are coprime to $N$, for any $\beta > 1$. For some $0 < \alpha < 1$, let $M$ be such that $Nk^{\alpha} \geq 2\pi M n$ and $M \geq 1$, where $n = e^{k/(k-3)}$ and $k >2$ is even. For the set $\mathcal{L}(\beta)$:

    \begin{equation*}
        \mathcal{L}(\beta) := \{ f \in H_k^*(N): Y_p(\theta_f(p)) \leq \alpha_p(0) - \delta_p, \  p \leq \beta, \  p \nmid N\},
    \end{equation*}

    and for any $\delta_p > 0$, we have,

    \begin{equation*}
    \sum_{f \in \mathcal{L}(\beta)} \omega_f \ll \left(1 + \frac{M^s}{Nk^{\eta}} \right) H^{-1},
    \end{equation*}

    where,

    \begin{equation*}
   H = \sum_{\substack{m \leq M \\ m| P_N(\beta)}} \prod_{p | m} \frac{\delta_p^2}{\sum_{1 \leq i \leq s} \alpha_p^2(i)}, \hspace{1em} P_N(\beta) = \prod_{\substack{p \leq \beta \\ p \nmid N}} p, \hspace{1em} \eta = k(1-\alpha) - \frac{k^{2\alpha-1}}{2}.
    \end{equation*}

    and $\omega_f$ is the harmonic weight,
    
    \begin{equation*}
         \omega_f =  \frac{\Gamma(k-1)}{(4\pi)^{k-1} \langle f,f \rangle}.
    \end{equation*}

    with $\langle \cdot, \cdot \rangle$ the Petersson inner product. For the $k=2$ case an extra factor of $\log (M^s)$ must be applied to the second term, and we have that $n=1$.
    \end{corollary}
For appropriately chosen $Y_p$ and $\delta_p$ for all $p \leq (\log kN)^A$, we see that $|S| \leq |\mathcal{L}((\log kN)^A)|$. Take the following degree 2 Chebyshev polynomial,

    \begin{equation*}
        F(\theta) =  - \frac{3}{4}X_0(\theta) + \frac{1}{2}X_1(\theta) + \frac{1}{4}X_2(\theta).
    \end{equation*}

    With this choice of $F$ we have the following,

    \begin{equation*}
        \int_0^{\pi} F(\theta) \textrm{ d}\mu_{ST} = -\frac{3}{4}, \hspace{1em} F(\theta) \leq \textrm{sgn}(2\cos \theta), \  \theta \in [0, \pi],
    \end{equation*}

    where $\textrm{sgn}(x)$ is the sign function. See Figure \ref{fig:Y_p} for reference. Furthermore, let $\tilde{F}(\theta) = F(\pi - \theta)$, where we assume at $\theta =\pi/2$ we have equality ($\tilde{F}(\pi/2) = F(\pi/2)$). Then for any sign sequence $(\varepsilon_p)_{p \leq \beta}$ define our $Y_p$ as,

    \begin{equation*}
        Y_p = \begin{cases}
            F &\textrm{ if } \varepsilon_p = 1, \\
            \tilde{F} &\textrm{ if } \varepsilon_p = -1. \\
        \end{cases}
    \end{equation*}

    If $I_p$ is the subset that $\theta_f(p)$ lies in so that $\lambda_f(p)\varepsilon_p > 0$, we can see that we have $\theta_f(p) \notin I_p \implies Y_p(\theta_f(p)) \leq -1$, and given that $\alpha_p(0) = -3/4$ for both $F$ and $\tilde{F}$ we have $\delta_p = 1/4$ for all $p \leq (\log kN)^A$ (with $p \nmid N$). Our choice of $Y_p$ implies that the size of the set $\mathcal{L}((\log kN)^A)$ upper bounds the size of our initial set $S$.
    
    Now we apply Corollary \ref{cor:weighted polynomial sieve} with $M = N^{1/2}k^{\gamma}$, where $1/2 < \gamma < 1$ and we pick $\alpha$ such that $ \gamma < \alpha < 1$ so that $ M \leq N k^{\alpha}/2\pi n$. Since we have that $\delta_p = 1/4$, $|\alpha_p(1)|= 1/2$ and $|\alpha_p(2)| = 1/4$ for all primes $p \leq (\log kN)^A$ and both $F$ and $\tilde{F}$, then the $H$ term in Corollary \ref{cor:weighted polynomial sieve} is just a sum over $m$ of $5^{-\omega(m)}$  where $\omega(n)$ counts the number of distinct prime factors of $n$. So for large enough $k ,N$ we have,

\begin{equation}
\label{eq:weighted sieve application}
    \sum_{f \in \mathcal{L}} \omega_f \ll \left(1+ \frac{\log (M^2)}{k^{\eta -2\gamma}}\right) \left( \sum_{\substack{m \leq M \\ m|P_N((\log kN)^A)}}  5^{-\omega(m)} \right)^{-1},
\end{equation}

where we include the $\log (M^s)$ to incorporate the $k=2$ case. We can see that we are summing over the set of square-free integers $\leq M$ that are $(\log kN)^A$-smooth\footnote{An integer is $y$-smooth if all its prime factors are less than or equal to $y$.} and coprime to $N$. We have bounds on the size of sets of this type. We label this set $\Psi_N^*(M, (\log kN)^A)$, then from \cite[p. 10]{DukKow},

\begin{equation*}
    |\Psi_N^*(M, (\log kN)^A) | \gg M^{1-1/A - \varepsilon},
\end{equation*}

with the implied constant depending on $\varepsilon$. We also note that any $m \in \Psi_N^*(M, (\log kN)^A)$ is necessarily less than a product of consecutive primes less than $M$, so that we have,

\begin{equation*}
    \omega(m) \ll \frac{\log M}{\log \log M},
\end{equation*}

then we also have,

\begin{equation*}
    \sum_{\substack{m \leq M \\ m|P_N((\log kN)^A)}} \exp(-\omega(m) \log 5) \gg \exp \left( -\log 5 \frac{\log M}{\log \log M} \right) M^{1- 1/A - \varepsilon} \gg M^{1- 1/A - 2\varepsilon},
\end{equation*}

for large enough $M$. We remark that the above bound on $H$ is still valid for any choice of $\alpha_p(i)$ and $\delta_p$, and thus valid for any choice of $Y_p$. We then have,

\begin{equation*}
    \sum_{f \in \mathcal{L}} \omega_f \ll \left(1 + \frac{\log (M^2)}{k^{\eta- 2\gamma}}\right)M^{-1+1/A+2\varepsilon},
\end{equation*}

and using $M = N^{1/2}k^{\gamma}$, then for large enough $k, N$,

\begin{align}
  \sum_{f\in \mathcal{L}} \omega_f \ll  N^{-1/2+1/2A + \varepsilon'}  k^{-\gamma + \gamma/A + 2\gamma \varepsilon}. \label{eq:gamma}
\end{align}

To remove the harmonic weight we have that the Petersson inner product can be related to the symmetric square $L$-function at $s=1$ (as in \cite[p. 138]{IwaKow}), and we have from \cite{HofLoc} and \cite{GolHofLie} that,

\begin{equation*}
    \frac{1}{\log kN} \ll L(\textrm{Sym}^2f, 1) \ll \log kN,
\end{equation*}

hence we have,

\begin{equation*}
    \frac{\Gamma(k)}{(4\pi)^k \log kN} N \log N \ll \langle f, f \rangle \ll \frac{\Gamma(k)}{(4 \pi)^k} N \log N \log kN,
\end{equation*}

so we may say,

\begin{equation}
\label{eq:omega_f}
    \omega_f \gg \frac{1}{kN \log N \log kN}.
\end{equation}

Combining \eqref{eq:gamma} and \eqref{eq:omega_f} we have, for any $\varepsilon_1 > 0$ and large enough $k$, $N$,

\begin{align}
   |S| \leq |\mathcal{L}((\log kN)^A)| &\ll (kN \log N \log kN) N^{-1/2+1/2A + \varepsilon'}  k^{-\gamma + \gamma/A + 2\gamma \varepsilon} \nonumber \\
   &\ll   N^{1/2 + 1/2A + \varepsilon_1}k^{1-\gamma + \gamma/A + 2\gamma \varepsilon'+\varepsilon''}, \label{eq:sieved S}
\end{align}

and we see that taking $\gamma$ close to $1$ optimises our bound. Then for any $\varepsilon_2 > 0$ we have,

\begin{equation*}
    |S| \ll_{\varepsilon_1, \varepsilon_2} N^{1/2+1/2A+\varepsilon_1}k^{1/A + \varepsilon_2},
\end{equation*}

for large enough $k$ and $N$, where the implied constant depends only on $\varepsilon_1$ and $\varepsilon_2$, which is exactly the statement in Theorem \ref{thrm: linnik nf}.

\begin{figure}[h]
\centering
\captionsetup{justification=centering}
\begin{tikzpicture}
    \begin{axis}[my style, xtick={(1/2)*pi, pi},
        xticklabel style={font=\footnotesize,fill=white},
        xticklabels={$\pi/2$,$\pi$}]
        \addplot[domain=0:pi, samples=100]{-1+ cos(deg(x)) + cos(deg(x))*cos(deg(x))}node[pos=0.9, below right]{$F$};
        \addplot[domain=0:pi, samples=100, blue]{-1-cos(deg(x)) + cos(deg(x))*cos(deg(x))}node[pos=0.9, above left, color=blue]{$\widetilde{F}$};
        \addplot[domain = pi/2:pi, samples = 100, red]{-1};
        \addplot[domain = 0:pi/2, samples = 100, red]{1}node[below, color=red]{sgn($2 \cos \theta$)};
    \end{axis}
\end{tikzpicture}
    \caption{$F  = - \frac{3}{4}X_0 + \frac{1}{2}X_1 + \frac{1}{4}X_2$,\\ $\widetilde{F}=-\frac{3}{4}X_0 + \frac{1}{2}X_1 - \frac{1}{4}X_2$, \\
    compared with $\textrm{sgn}(2 \cos \theta)$ in red.}
\label{fig:Y_p}
\end{figure}
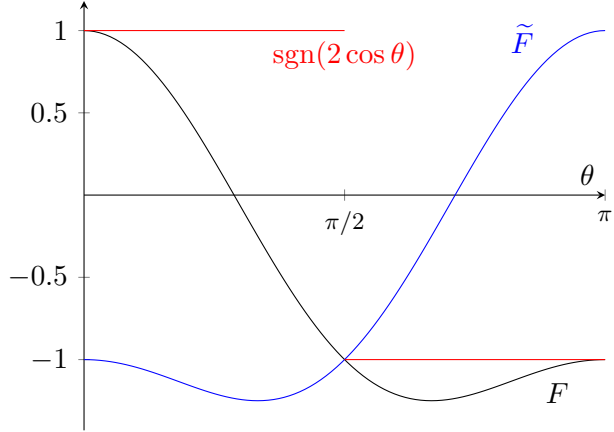

\subsection{Proof of Theorem \ref{thrm:weighted linnik sieve}}
\label{sec:proof linnik sieve}

This weighted sieve is derived from the Petersson trace formula,

\begin{equation*}
    \frac{\Gamma(k-1)}{(4 \pi \sqrt{mn})^{k-1}} \sum_{f \in \mathcal{F}} a_f(n) \overline{a_f(m)} = \delta(m,n) + 2 \pi i^{-k} \sum_{\substack{c>0 \\ c \equiv 0 \text{ (mod } N)}} c^{-1} S(m,n;c)J_{k-1} \left( \frac{4\pi \sqrt{mn}}{c} \right),
\end{equation*}

for $\mathcal{F}$ an orthonormal basis to $S_k(N)$. We have $S(m,n;c)$ is the Kloosterman sum, and $J_{k-1}(2x)$ is the Bessel function of the first kind. The $a_f(n)$ refer to the orthonormalised Fourier coefficients of the form $f \in \mathcal{F}$. As such, let's consider the following weighted Fourier coefficients,

\begin{equation}
\label{eq:alg norm}
    \psi_f(n) = \left( \frac{\Gamma(k-1)}{(4 \pi n)^{k-1}} \right)^{1/2} a_f(n).
\end{equation}

Then complex numbers $a_m$ we look to find, via the Petersson trace formula, $\Delta(N,k,M)$ so that,

\begin{equation}
\label{eq:sieve}
    \sum_{f \in \mathcal{F}} \left| \sum_{m \leq M} a_m \psi_f(m) \right|^2 \ll \Delta(N,k,M) \|a\|^2,
\end{equation}

where we may achieve sufficient savings in $k$ and $N$. From the Petersson trace formula we can see that the left-hand side of \eqref{eq:sieve} is equal to,

\begin{equation}
\label{eq:first step}
    \| a\|^2 + 2 \pi i^{-k} \sum_{\substack{c>0 \\ c \equiv 0 \text{ (mod } N)}} c^{-1} \mathop{\sum\sum}_{m,n \leq M}\bar{a}_ma_n S(m,n;c) J_{k-1}\left( \frac{4 \pi \sqrt{mn}}{c} \right).
\end{equation}

By Cauchy-Schwarz on the Kloosterman sum, followed by an application of the classical large sieve we have that the Kloosterman sum alone may be bounded as follows.

\begin{align*}
    \left| \mathop{\sum\strut^*}_{x \text{ (mod } c)}\mathop{\sum\sum}_{m,n \leq M} e \left( \frac{xm + \bar{x}n}{c} \right) \bar{a}_ma_n \right|^2 &\leq \sum_x \left| \sum_{m \leq M} e \left( \frac{xm}{c} \right) \bar{a}_m \right|^2 \cdot \sum_{x} \left| \sum_{n \leq M} e \left(\frac{\bar{x}n}{c} \right) a_n \right|^2\\
    &\leq \left( (c+M) \| a\|^2 \right)^2.
\end{align*}

Alternatively, we may apply Weil's bound on the Kloosterman sum:

\begin{equation*}
    |S(m,n;c)| \leq (m,n,c)^{1/2}c^{1/2} \tau(c),
\end{equation*}

so that,

\begin{align*}
    \left| \mathop{\sum\sum}_{m,n \leq M} \bar{a}_m a_n S(m,n;c) \right| &\leq \mathop{\sum\sum}_{m,n \leq M} |\bar{a}_m a_n| (m,n,c)^{1/2} c^{1/2} \tau(c) = \tau(c) c^{1/2}\sum_{d|c} d^{1/2} \mathop{\sum\sum}_{m,n \leq M/d} |\bar{a}_{md} a_{nd}|\\
    &\leq \tau(c) c^{1/2} \sum_{d|c} d^{1/2} \left( \sum_{m \leq M/d} |b_{md}| \right)^2  \\
    &\leq \tau(c) \sum_{d|c} d^{1/2} \left( \sum_{m \leq M/d} 1 \right)\left( \sum_{m \leq M/d} |b_{md}|^2 \right) \\
    &\leq \tau(c)^2 c^{1/2} M \|a\|^2.
\end{align*}

Then depending on the size of $c$ we may achieve more savings through Weil's bound. Recall in the inner sum of \eqref{eq:first step} we are twisting by the Bessel function, and we denote this part of the sum as $L(c)$:

\begin{equation}
    L(c) = \mathop{\sum\sum}_{m,n\leq M} \bar{a}_ma_n S(m,n;c) J_{k-1} \left( \frac{4 \pi \sqrt{mn}}{c} \right),
    \label{eq:L(c)}
\end{equation}

where the Bessel function has the following power series expansion,

\begin{equation*}
    J_{k-1}(2x) = \sum_{l \geq 0} \frac{ (-1)^l x^{k-1+2l}}{l!\Gamma(k+l)},
\end{equation*}

then using the bounds we have on the Kloosterman sum,

\begin{align}
    \left| L(c) \right|^2 \leq \sum_{l \geq 0} \left| \frac{(-1)^l (2 \pi)^{k-1+2l}}{l!\Gamma(k+l)c^{k-1+2l}} \right|^2 \left| \mathop{\sum\sum}_{m,n \leq M}\bar{a}_m a_n S(m,n;c) (mn)^{\frac{k-1+2l}{2}}\right|^2 \nonumber \\ 
    \leq \sum_{l \geq 0} \left| \frac{(-1)^l}{l! \Gamma(k+l)} \cdot \left( \frac{2\pi M}{c} \right)^{k-1+2l} \right|^2 \min \{(c+M)^2, c\tau(c)^4M^2\} \|a\|^4.
    \label{eq:Kloosterman and bessel}
\end{align}

Denote the final sum as $\mathscr{J}_{k-1}(x)$ so that,

\begin{equation*}
    \mathscr{J}_{k-1}(x) = \sum_{l \geq 0} \left| \frac{(-1)^l x^{k-1+2l}}{l!\Gamma(k+l)} \right|^2 = \sum_{l \geq 0} \left( \frac{x^{k-1+2l}}{l! \Gamma(k+l)} \right)^2.
\end{equation*}

To ensure some saving in $k$ we impose the condition $|x| = | 2\pi M/c | \leq k^{\alpha}/n$ for some $\alpha, n > 0$. This is satisfied for any $c > 0$ with $c \equiv 0$ (mod $N$) if $Nk^{\alpha} \geq 2\pi Mn$. If we first have $k \geq 4$ (and even) we may say,

\begin{align*}
    \mathscr{J}_{k-1}(x) &\leq x^4 \left( \frac{k^{\alpha(k-3)}}{n^{k-3}\Gamma(k)}\right)^2 + x^4 \sum_{l \geq 1} \left( \frac{k^{\alpha(k-3+2l)}}{l! \Gamma(k+l)} \right)^2  \\
    &\leq x^4 \left(\frac{k^{\alpha(k-3)}e^k}{n^{k-3}k^{k-1/2}} \right)^2+ x^4\sum_{l \geq 1} \left( \frac{k^{\alpha(k-3+2l)}e^{l+k}}{n^{k-3+2l}l^{l+1/2}(k+l)^{k+l-1/2}} \right)^2,
\end{align*}

via Stirling's approximation. With $n = e^{k/(k-3)}$ we can see that we have,

\begin{align*}
    \mathscr{J}_{k-1}(x) &\leq x^4(k^{k(\alpha-1)-3\alpha+1/2})^2\left( 1 + \sum_{l \geq 1} \left( \frac{k^{k+2\alpha l-1/2}}{ l^{l+1/2}(k+l)^{k+l-1/2}} \right)^2 \right) \\
    &\leq x^4(k^{k(\alpha-1)-3\alpha+1/2})^2 \left(1 + \sum_{l \geq 1} \left( \frac{k^{2\alpha l}}{{l^{l+1/2}(k+l)^{l}}} \right)^2  \right).
\end{align*}

To achieve any savings in $k$, we can see from the factor of $k^{k(\alpha-1)-3\alpha + 1/2}$ that we require $\alpha \leq 1$. Further, we can see that the $l$-sum clearly converges to a constant for all $\alpha \leq 1/2$. So let $1/2 < \alpha \leq 1$. First, we can see that for $l > k^{2\alpha -1}$ the sum is suitably small.

\begin{equation*}
    \sum_{l >k^{2\alpha -1}} \frac{k^{4\alpha l}}{l^{2l+1}(k+l)^{2l}} < \sum_{l > k^{2\alpha -1}} \frac{k^{4\alpha l}}{(k^{2\alpha -1})^{2l + 1} (k(1+k^{2\alpha-2}))^{2l}} = \frac{1}{k^{2\alpha -1}} \sum_{l >k^{2\alpha -1}} \frac{1}{(1+k^{2\alpha -2})^{2l}}.
\end{equation*}

On the other hand, if we let $l = xk^{2\alpha -1}$ where $x = i/k^{2\alpha- 1}$ for $i = 1, ..., \lfloor k^{2 \alpha -1} \rfloor$,

\begin{equation*}
    \sum_{l \leq k^{2\alpha-1}} \frac{k^{4\alpha l}}{l^{2l+1}(k+l)^{2l}} = \sum_{i=1}^{\lfloor k^{2\alpha -1} \rfloor} \frac{k^{4\alpha xk^{2\alpha-1}}}{(xk^{2\alpha-1})^{2xk^{2\alpha-1}+1}(k+xk^{2\alpha-1})^{2xk^{2\alpha-1}}}.
\end{equation*}

Taking logarithms we see for each term, and noting $1/k^{2\alpha-1} \leq x \leq 1$,

\begin{multline*}
        \log \left( \frac{k^{4\alpha xk^{2\alpha -1}}}{(xk^{2\alpha-1})^{2xk^{2\alpha-1}+1}(k+xk^{2\alpha-1})^{2xk^{2\alpha-1}}} \right)= \\
        -(2\alpha -1 ) \log k - \log x -2xk^{2\alpha -1}\log x  - 2xk^{2\alpha -1}\log(1+xk^{2\alpha -2}) \\
        \leq -2xk^{2\alpha-1}\log x.
\end{multline*}

By considering the derivatives, we can see $-2xk^{2\alpha-1}\log x$ is maximised at $x = e^{-1}$, hence we have,

\begin{equation*}
    \sum_{l \leq k^{2\alpha -1}} \frac{k^{4\alpha l}}{l^{2l+1}(k+l)^{2l}} \leq k^{2\alpha -1} \exp (k^{2\alpha -1}).
\end{equation*}

Then for all $1/2 < \alpha \leq 1$ with $k \geq 4$ we have,

\begin{equation}
    \mathscr{J}_{k-1}(x) \ll x^4 \left(k^{k(\alpha-1) -2\alpha} \exp \left(\frac{k^{2\alpha -1}}{2}\right)\right)^2 \leq x^4 (k^{k(\alpha-1)-2\alpha + \frac{k^{2\alpha-1}}{2}})^2.
    \label{eq:curly J bound}
\end{equation}

On the other hand, for $0 < \alpha \leq 1/2$ we have a slightly different bound. For our purposes we require $\alpha > 1/2$, hence we omit this. See Remark \ref{rem: alpha} for more detail.

Writing $\eta_0 = k(\alpha - 1) -2\alpha +k^{2\alpha-1}/2$ we use this bound on $\mathscr{J}_{k-1}(x)$ in \eqref{eq:Kloosterman and bessel}, and we note that 
the bound on the Kloosterman sum obtained through the classical large sieve is sufficient here. Recall $x= 2\pi M/c$, then $L(c)$ in \eqref{eq:L(c)} may be bounded by,

\begin{equation*}
    L(c) \leq (c+M) \left( \frac{2\pi M}{c} \right)^2 k^{\eta_0} \|a\|^2.
\end{equation*}

We use this in \eqref{eq:first step} to determine $\Delta(N,k,M)$ for $Nk^{\alpha} \geq 2\pi Mn$ with $k \geq 4$,

\begin{align}
    \Delta(N,k,M) \| a \|^2 &= \|a\|^2 + 2\pi k^{\eta_0}\sum_{\substack{c>0 \\ c \equiv 0 \text{ (mod }N)}} c^{-1} \left( \frac{2\pi M}{c} \right)^2 (c+M)\|a\|^2 \nonumber\\
    &= \|a\|^2 + \frac{ 2 \pi k^{\eta_0} \|a\|^2}{N} \sum_{d \geq 1} \frac{(2\pi M)^2}{d^2N} + \frac{(2 \pi)^2 M^3}{d^3N^2} \nonumber\\
    &\leq \|a\|^2 + \frac{2 \pi k^{\eta_0} \|a\|^2}{N} \sum_{d\geq 1} \frac{2\pi M k^{\alpha}}{n d^2} + \frac{Mk^{2\alpha}}{n^2 d^3} \nonumber\\
    &\ll \|a\|^2 \left(1 + \frac{M}{Nk^{k(1-\alpha) - k^{2\alpha - 1}/2}} \right).\label{eq:Delta}
\end{align}

In the case $k=2$ we may not use the same bound \eqref{eq:curly J bound} on $\mathscr{J}_1(x)$. With $x < 2^{\alpha}$ we may only say,

\begin{align*}
    \mathscr{J}_1(x) = x^2\left( 1+ \sum_{l \geq 1} \left( \frac{x^{2l}}{l!(l+1)!} \right)^2 \right) \ll x^2.
\end{align*}

However we have that for large $c$, Weil's bound on the Kloosterman sum gives us an improvement, thus the sum over $c$ in \eqref{eq:first step} can be bounded by,

\begin{equation*}
    \smashoperator{\sum_{\substack{c>0 \\ c\equiv 0 (\textrm{ mod }N)}}} \frac{2\pi M}{c^2} \min \{ c+M, c^{1/2} \tau(c)^2 M \} \| a\| ^2 \leq \smashoperator{\sum_{\substack{0 < c < M^2 \\ c\equiv 0 \textrm{ (mod } N)}}} \frac{2\pi M}{c^2} (c+M)\|a\|^2 +\smashoperator{\sum_{\substack{c > M^2 \\ c\equiv 0  \textrm{ (mod } N)}}} \frac{2\pi M}{c^{3/2}} \tau(c)M \|a\|^2.
\end{equation*}

We see that an extra factor of $\log M$ occurs in the first sum,

\begin{equation*}
  \smashoperator{\sum_{\substack{0 < c < M^2 \\ c\equiv 0 \textrm{ (mod } N)}}} \frac{2\pi M}{c^2} (c+M) \leq \frac{2\pi M}{N} \sum_{d< M^2/N} \frac{1}{d} + \frac{2^{\alpha-1}}{d^2 \pi} \ll \frac{M\log M}{N},
\end{equation*}

and using the well known estimate $\tau(n) \sim n^{\varepsilon}$ we see that the tail sum is appropriately small,

\begin{equation*}
    \smashoperator{\sum_{\substack{c > M^2 \\ c\equiv 0  \textrm{ (mod } N)}}} \frac{2\pi M}{c^{3/2}} \tau(c)M \leq \frac{2\pi M^2}{N^{3/2-\varepsilon}} \sum_{d >M^2/N} \frac{1}{d^{3/2-\varepsilon}} < \frac{2\pi M^2}{N^{3/2-\varepsilon}} \int_{M^2/N}^{\infty} t^{-3/2+\varepsilon} \textrm{ d}t \ll \frac{M}{N}.
\end{equation*}

Thus for $k \geq 4$ and even, we have the general large sieve over some orthnoromal basis $\mathcal{F}$:

\begin{equation}
    \sum_{f \in \mathcal{F}} \left| \sum_{m \leq M} a_m \psi_f(m) \right|^2 \ll \left( 1 + \frac{M}{Nk^{\eta}} \right)^2 \| a \|^2,
\end{equation}

provided $Nk^{\alpha} \geq 2\pi Mn$, and where $\eta = k(1-\alpha) - k^{2\alpha -1}/2$. And in the $k=2$ case we have an extra factor of $\log M$ in the second term. Using this we may easily infer a statement about Hecke eigenvalues, $\lambda_f(m)$, as we have,

\begin{equation*}
   \left( \frac{\Gamma(k-1)}{(4 \pi m)^{k-1}} \right)^{-1/2} \psi_f(m) = a_f(m) = m^{\frac{k-1}{2}}\lambda_f(m).
\end{equation*}

Then with some orthonormal basis $\mathcal{B}$ we have,

\begin{equation*}
    \mathcal{B} = \left\{ \frac{f}{\langle f,f \rangle^{1/2}} : f \in S_k(N) \right\} \supseteq \left\{ \frac{f}{\langle f,f \rangle^{1/2}} : f \in H_k^*(N) \right\}. 
\end{equation*}

Hence for complex numbers $a_m$ that are 0 if $(m,N) \neq 1$,

\begin{equation*}
    \sum_{f \in H_k^*(N)} \frac{\Gamma(k-1)}{(4 \pi)^{k-1} \langle f,f \rangle} \left| \sum_{\substack{m \leq M \\ m \nmid N}} a_m \lambda_f(m) \right|^2 \ll \left(1 + \frac{M}{Nk^{\eta}} \right) \| a\|^2,
\end{equation*}

as required.

\begin{remark}
\label{rem: alpha}
We may remove the restriction on the length of the sieved sum to give a slightly less strong sieve inequality, see Theorem \ref{thrm:complete linnik sieve}. We note that at $\alpha = 1$ we actually have a loss in $k$ thus in Theorem \ref{thrm:weighted linnik sieve} we restrict to $1/2 < \alpha <1$. We note that for $\alpha \leq 3/4$ we achieve savings in $k$ for all $k > 2$, while for $3/4 < \alpha <1$ we only achieve a saving for large enough $k$. On the other hand, if we restrict to $0 < \alpha \leq 1/2$ we may very marginally improve the bound in $k$ through a different bound on $\mathscr{J}_{k-1}(x)$, but this is omitted here as it does not change the overall bound when the condition $Nk^{\alpha} \geq 2\pi M n$ is removed.
\end{remark}

\subsection{Proof of Corollary \ref{cor:weighted polynomial sieve}}

To show this corollary we will make use of an amplification trick. This is a well-known trick in analytic number theory and has been used in many areas, most relevantly in improving the classical large sieve (see \cite[Section 7]{Mon}). The idea is to apply an amplifier to our sieved set in the hopes that the new amplified object will produce a stronger bound when applying our large sieve Theorem \ref{thrm:weighted linnik sieve}. For any $f \in \mathcal{L}$ we have,

\begin{align}
        Y(\theta_f(p)) &\leq \alpha_p(0) - \delta_p \nonumber \\
        \implies \delta_p &\leq - \sum_{i=1}^{s} \alpha_p(i) X_i(\theta_f(p)). \label{eq:delta}
\end{align}

Recall we have the property of Hecke multiplicity which can be seen in the Chebyshev polynomials \eqref{eq:hecke chebyshev mult}, so that we may rewrite \eqref{eq:delta} in terms of Hecke eigenvalues. Consider taking a product over prime divisors of some integer $m$, where $m$ is square-free, $\beta$-smooth and co-prime to $N$. Then any such $m$ is simply a divisor of $P_N(\beta)$. That is, $m = p_1  p_2 \dots  p_l$ (for example) where for each $p_i$ we have $p_i \leq \beta$, $p_i \nmid N$ and $p_i \neq p_j$ for all $i \neq j$. Expanding the sum and product, and applying Chebyshev multiplicativity, we have,

\begin{equation}
\label{eq:Y_p hecke eigenvalues}
    \prod_{p | m} \sum_{i=1}^s -\alpha_p(i) X_i(\theta_f(p)) = \sum_{\substack{d|m^s  \\m|d}} \left( \prod_{p|d} -\alpha_p(v_p(d)) \right)\lambda_f(d),
\end{equation}

where $v_p(d)$ is the $p$-adic valuation of $d$. The sum over divisors $d$ of $m^s$ such that $m$ still divides $d$ are integers of the type $\prod_{p | m} p^j$ such that $j \neq 0$ (and clearly $j \leq s$). Consider now taking a sum over all such $m$ subject to $m \leq M$ for some $M > 1$ and $m | P_N(\beta)$, then from \eqref{eq:delta} and \eqref{eq:Y_p hecke eigenvalues} we see,

\begin{equation*}
    \frac{\sum_{\substack{m \leq M \\ m | P_N(\beta)}}\sum_{\substack{d |m^s \\ m|d}} \left( \prod_{p|d} - \alpha_p(v_p(d)) \right) \lambda_f(d)}{ \sum_{\substack{m \leq M \\ m | P_N(\beta)}}\prod_{p|m} \delta_p} \geq 1.
\end{equation*}
    
By construction of $P_N(\beta)$, the sum over $m$ is over square-free integers. Hence, after squaring, we may say,

\begin{equation*}
    \sum_{f \in \mathcal{L}(\beta)} \omega_f \leq \sum_{f \in H_k^*(N)}  \omega_f \frac{\left| \sum_{\substack{m \leq M \\ m | P_N(\beta)}}\sum_{\substack{d |m^s \\ m|d}} \left( \prod_{p|d}  (- \alpha_p(v_p(d))) \right)\lambda_f(d)\right|^2}{\left|\sum_{\substack{m \leq M \\ m | P_N(\beta)}}\prod_{p|m} \delta_p \right|^2}.
\end{equation*}

At this point we introduce our amplifiers to minimize the loss in this bound. Let $\xi_p$ be arbitrary positive, real numbers, so that we still have,

\begin{equation*}
    \frac{\sum_{\substack{m \leq M \\ m | P_N(\beta)}} \sum_{\substack{d |m^s \\ m|d}} \prod_{p|d} (  - \alpha_p(v_p(d)) )\xi_p\lambda_f(d)  }{\sum_{\substack{m \leq M \\ m | P_N(\beta)}}\prod_{p|m}\xi_p \delta_p} \geq 1.
\end{equation*}

Note in the numerator we may move the product of $\xi_p$ over primes that divide $m$ into the product over primes that divide $d$ as by construction these primes are the same. Then we apply Theorem \ref{thrm:weighted linnik sieve} with,

\begin{equation}
    a_d = \begin{cases}
       \prod_{p|d}   (-\alpha_p(v_p(d))) \xi_p &\text{if } \exists \  m\leq M, m|P_N(\beta) \text{ s.t. } d|m^s \text{ and } m|d \\
        0 &\text{otherwise.}
    \end{cases}
    \label{eq:a_d}
\end{equation}

For each $d$ there is only one such $m$ that satisfies the above conditions, so we may consider the double sum over $m$ and $d$ as one sum over $d$, subject to the conditions in \eqref{eq:a_d}. Then we see that the largest our index $d$ can be is $M^s$, so by Theorem \ref{thrm:weighted linnik sieve} on the numerator alone we have,

\begin{multline*}
   \frac{\sum_{f \in H_k^*(N)}  \omega_f\left| \sum_{\substack{m \leq M \\ m | P_N(\beta)}}\sum_{\substack{d |m^s \\ m|d}} \left( \prod_{p|d}  (- \alpha_p(v_p(d))) \xi_p\right)\lambda_f(d)\right|^2}{\left|\sum_{\substack{m \leq M \\ m | P_N(\beta)}}\prod_{p|m} \delta_p \xi_p\right|^2} \ll \\
    \left( 1 + \frac{M^s}{Nk^{\eta}} \right) \frac{\sum_{\substack{m \leq M \\ m| P_N(\beta)}} \sum_{\substack{d|m^s \\ m|d}} \left| \prod_{p|d} \alpha_p(v_p(d)) \xi_p \right|^2}{\left|\sum_{\substack{m \leq M \\ m | P_N(\beta)}}\prod_{p|m} \delta_p \xi_p\right|^2}
\end{multline*}

We find we achieve the most cancellation if we set $\xi_p = \frac{\delta_p}{\sum_{1 \leq i \leq s} \alpha_p(i)^2}$. Then in the numerator we have via \eqref{eq:Y_p hecke eigenvalues},

\begin{multline*}
  \sum_{\substack{m \leq M \\ m|P_N(\beta)}} \sum_{\substack{d|m^s \\ m|d}} \prod_{p|d} \alpha_p(v_p(d))^2 \frac{\delta_p^2}{\left( \sum_{1 \leq i \leq s} \alpha_p(i)^2 \right)^2} = \\
  \sum_{\substack{m\leq M \\ m|P_N(\beta)}} \prod_{p|m} \sum_{1 \leq i\leq s} \alpha_p(i)^2 \frac{\delta_p^2}{\left(\sum_{1 \leq i \leq s} \alpha_p(i)^2 \right)^2} = \sum_{\substack{ m \leq M \\ m|P_N(\beta)}}\prod_{p |m} \frac{\delta_p^2}{\sum_{1 \leq i \leq s} \alpha_p(i)^2} .
\end{multline*}

While in the denominator we have,

\begin{equation*}
    \left| \sum_{\substack{m \leq M \\ m | P_N(\beta)}}\prod_{p|m} \delta_p \xi_p\right|^2 = \left( \sum_{\substack{m \leq M \\ m | P_N(\beta)}} \prod_{p|m} \frac{\delta_p^2}{\sum_{1 \leq i \leq s} \alpha_p(i)^2} \right)^2,
\end{equation*}

giving the desired result,

\begin{equation*}
    \sum_{f \in \mathcal{L}(\beta)} \omega_f \ll \left( 1 + \frac{M^s}{Nk^{\eta}} \right) \left( \sum_{\substack{ m \leq M \\ m|P_N(\beta)}} \prod_{p|m} \frac{\delta_p^2}{\sum_{1 \leq i \leq s} \alpha_p(i)^2} \right)^{-1}.
\end{equation*}

Clearly we see in the $k=2$ case the only change is the additional factor of $\log (M^s)$ that occurs when applying Theorem \ref{thrm:weighted linnik sieve}.

\appendix

\section{A complete sieve}

We complete Theorem \ref{thrm:weighted linnik sieve} by removing the requirement $Nk^{\alpha} \geq 2\pi Mn$. We may remove this condition via dimension analysis, however this reduces the overall bound.

\begin{theorem}
\label{thrm:complete linnik sieve}
    Let $\mathcal{F}$ be an orthonormal basis of $S_k(N)$. Let $a_m$ be complex numbers, with $M > 1$ and $\psi_f(m)$ as in \eqref{eq:alg norm}. We have,

    \begin{equation*}
        \sum_{f \in \mathcal{F}} \left| \sum_{m \leq M} a_m \psi_f(m) \right|^2 \ll_{\varepsilon} \left( 1 + \frac{M}{Nk^{1-\varepsilon}} \right) \|a\|^2.
    \end{equation*}

    for any $0 < \varepsilon < 1$ where the implied constant depends only on $\varepsilon$. For $k = 2$ we instead have,
    
    \begin{equation*}
        \sum_{f \in \mathcal{F}} \left| \sum_{m \leq M} a_m \psi_f(m) \right|^2 \ll \left( 1 + \frac{M\log M}{N} \right) \|a\|^2.
    \end{equation*}
    
\end{theorem}

\begin{proof}
We have that $[\Gamma_0(1): \Gamma_0(N)] = N \prod_{p|N}(1+p^{-1})$, the index of $\Gamma_0(N)$ in $\Gamma_0(1)$. Then for a prime $p$ we have,

\begin{equation*}
    [\Gamma_0(N): \Gamma_0(Np)] = p \prod_{\substack{q|p \\ q \nmid N}}(1+q^{-1}) \leq p+1.
\end{equation*}

Further, we may write the volume of the fundamental domain of $\Gamma_0(Np)$ in terms of that of $\Gamma_0(N)$,

\begin{equation*}
    Vol(\Gamma_0(Np) \backslash \mathbb{H} ) = Vol(\Gamma_0(N) \backslash \mathbb{H}) [\Gamma_0(N): \Gamma_0(Np)].
\end{equation*}

Hence if $\mathcal{F}_L$ denotes an orthonormal basis of $S_k(L)$, we have,

\begin{align*}
    \sum_{f \in \mathcal{F}_{Np}} \left| \sum_{m \leq M} a_m \psi_f(m)\right|^2 &= \sum_{f\in \tilde{\mathcal{F}}_{N}} \left| \sum_{m \leq M} a_m \psi_f(m)\right|^2 + \sum_{f\in F_{Np} \backslash \tilde{\mathcal{F}_N}}  \left| \sum_{m \leq M} a_m \psi_f(m)\right|^2 \\
    &\geq \sum_{f \in \tilde{\mathcal{F}}_N} \left| \sum_{m \leq M} a_m \psi_f(m) \right|^2 \\
    &\geq \frac{1}{p+1}\sum_{f \in \mathcal{F}_N} \left| \sum_{m \leq M} a_m \psi_f(m) \right|^2,
\end{align*}

where $\tilde{\mathcal{F}}_N \subseteq \mathcal{F}_{Np}$ is the set of forms in $\mathcal{F}_{Np}$ that also have level $N$. For any $Nk^{\alpha} < 2\pi M n$ let $p$ be a prime such that $2\pi Mn \leq Npk^{\alpha} \leq 4 \pi Mn$. We have from \eqref{eq:Delta},

\begin{align*}
    \sum_{f \in \mathcal{F}_N} \left| \sum_{m \leq M} a_m \psi_f(m) \right|^2 &\leq (p+1)  \sum_{f \in \mathcal{F}_{Np}} \left| \sum_{m \leq M} a_m \psi_f(m)\right|^2 \\
    &\ll (p+1) \left( 1 + \frac{M}{Npk^{\eta}} \right) \| a \|^2 \\
    &\ll \left(1 + \frac{M}{Nk^{\alpha}} + \frac{M}{Nk^{\eta}} \right)\|a\|^2
\end{align*}

Recall $\eta = k(1-\alpha) - k^{2\alpha -1}/2$, thus for $1/2 < \alpha < 1$ we have for large enough $k$,

\begin{equation*}
     \sum_{f \in \mathcal{F}_N} \left| \sum_{m \leq M} a_m \psi_f(m) \right|^2 \ll \left(1 + \frac{M}{Nk^{\alpha}} \right) \|a\|^2.
\end{equation*}

In fact we may say the same for $0 < \alpha \leq 1/2$. As mentioned in Remark \ref{rem: alpha} we may improve the saving in $k$ in \eqref{eq:Delta} for small $\alpha$, but overall we achieve the same saving as above. The $k=2$ case follows similarly. One may also make an analogous statement in terms of Hecke eigenvalues, as in Theorem \ref{thrm:weighted linnik sieve}.

\end{proof}

\newpage
\printbibliography

\end{document}